\documentclass[11pt]{amsart}
\usepackage[run=2]{auto-pst-pdf}
\usepackage{amsmath}
\usepackage{amssymb}
\usepackage{amscd}
\usepackage{url}
\usepackage{xypic}
\usepackage{pstricks}
\usepackage{pst-node}
\usepackage{pst-coil}
\usepackage{graphicx}

\DeclareMathOperator{\MHS}{MHS}

\DeclareMathOperator{\Hom}{Hom}
\DeclareMathOperator{\lk}{lk}
\DeclareMathOperator{\clk}{flk}

\DeclareMathOperator{\im}{im}

\DeclareMathOperator{\Irr}{Irr}

\DeclareMathOperator{\Sp}{Sp}

\DeclareMathOperator{\coker}{coker}

\DeclareMathOperator{\Var}{Var}
\DeclareMathOperator{\Gr}{Gr}

\numberwithin{equation}{section}
\numberwithin{equation}{subsection}

\theoremstyle{plain}
\newtheorem{theorem}[equation]{Theorem}
\newtheorem{lemma}[equation]{Lemma}
\newtheorem{proposition}[equation]{Proposition}

\newtheorem{corollary}[equation]{Corollary}

\theoremstyle{definition}
\newtheorem{example}[equation]{Example}
\newtheorem{remark}[equation]{Remark}

\newtheorem{definition}[equation]{Definition}
\newtheorem{convention}[equation]{Convention}

\numberwithin{equation}{subsection}

\def\bp{\begin{pmatrix}}
\def\ep{\end{pmatrix}}
\def\be{\begin{equation}}
\def\ee{\end{equation}}
\def\fix{\rm{fix}}
\def\frct{\rm{frct}}

\def\mend{\rm{mend}}
\def\nullk{\rm{null}}

\newcommand{\ol}[1]{\overline{#1}}
\newcommand{\mv}{\mathcal{V}}
\newcommand{\mw}{\mathcal{W}}

\newcommand{\mvn}{\mv_{\neq 1}}
\newcommand{\mvi}{\mv_{\im}}
\newcommand{\mvb}{\mv_{\partial}}
\newcommand{\mvf}{\mv_{\fix}}
\newcommand{\mvc}{\mv_{\frct}}
\newcommand{\mvm}{\mv_{\mend}}

\newcommand{\Spc}{\Sp_{\frct}}
\newcommand{\SpM}{\Sp_{\MHS}}

\newcommand{\Spf}{\Sp_{\fix}}

\newcommand{\Spn}{\Sp_{\neq 1}}
\newcommand{\bn}{b_{\neq 1}}
\newcommand{\Sn}{S_{\neq 1}}
\newcommand{\Si}{S_{\im}}
\newcommand{\Sb}{S_\partial}
\def\vi{U_{\im}}

\def\vb{U_{\partial}}
\def\vbb{U_B}
\def\ve{U_{=1}}
\def\vnu{U_{\nullk}}
\def\vs{U^\S}
\def\vn{U_{\neq 1}}
\def\vf{U_{\fix}}

\def\p{\partial}
\def\C{\mathbb{C}}
\def\R{\mathbb{R}}
\def\e{\varepsilon}
\def\ef{\epsilon}
\def\S{\Sigma}
\def\O{\Omega}
\def\l{\lambda}
\def\s{\sigma}
\def\sm{\setminus}
\def\X{\mathcal{X}}

\def\Z{\mathbb Z}
\def\Q{\mathbb Q}

\def\hs{h^\S}
\def\hst{h^\S_t}

\def\ge{\geqslant}
\def\le{\leqslant}
\newrgbcolor{darkgreen}{0.2 0.6 0.2}
\newrgbcolor{moderateblue}{0.5 0.5 1.0}
\newrgbcolor{crimson}{0.6 0.1 0.1}
\title[The spectrum of analytic germs on surface singularities]{The Hodge spectrum
of analytic germs on isolated surface singularities}
\author{Maciej Borodzik}
\address{Institute of Mathematics, University of Warsaw, ul. Banacha 2,
02-097 Warsaw, Poland}
\email{mcboro@mimuw.edu.pl}
\thanks{The first author is supported by Polish 
OPUS grant No 2012/05/B/ST1/03195.
The second author is partially supported by OTKA Grants 100796 and K67928. }
\author{Andr\'as N\'emethi}
\address{A. R\'enyi Institute of Mathematics, 1053 Budapest,  Re\'altanoda u. 13-15,  Hungary.}
\email{nemethi@renyi.hu}

\date{\today}
\makeatletter
\@namedef{subjclassname@2010}{\textup{2010} Mathematics Subject Classification}
\makeatother

\begin{document}

\begin{abstract}
We use topological methods to prove a semicontinuity property
 of the Hodge  spectra for analytic germs defined on an isolated surface singularity.
For this we introduce an analogue of the Seifert matrix (the fractured Seifert matrix),
and  of the Levine--Tristram signatures associated with it,
defined for null-homologous links in arbitrary three dimensional manifolds. Moreover, we establish
Murasugi type inequalities in the presence of cobordisms of links.

It turns out that the fractured Seifert matrix determines the Hodge spectrum and the
Murasugi type inequalities can be read as spectrum semicontinuity inequalities.
\end{abstract}
\maketitle

\begin{flushright}{\it Dedicated to Joseph Steenbrink}\end{flushright}

\vspace{5mm}

\section{Introduction}
In a series of articles \cite{BN,BN2} (see also \cite{BNR,BNR2}) the authors
developed a topological method to prove the semicontinuity of the Hodge spectrum in low dimensions,
which originally was obtained by purely  Hodge theoretical methods (that is, by algebraic or analytic
machinery).
This topological method worked successfully for local plane curve singularities, or for two-variable
complex polynomials (for the mixed Hodge structure at infinity). It was even possible
to compare the spectrum at infinity with local spectra of singular points of a fixed fiber of a
polynomial map. The approach used only topological, not analytic, arguments;
the idea was that upon using the  polarization properties of the mixed Hodge structure,
the spectrum was characterized by the Levine--Tristram signatures  of the Seifert form. Next,
 in the presence of a deformation, using the corresponding topological cobordism one proved
a Murasugi type inequality valid for the Levine--Tristram signatures. This inequality was reinterpreted as a spectrum
semicontinuity property.

Having these results, it was natural to ask if these method can be generalized to
germs $g:(X,0)\to (\C,0)$ defined on an arbitrary isolated surface singularities $(X,0)$;
in fact, this question was asked explicitly
by J. Steenbrink at the meeting in Lille in 2012 during a presentation of the  first author.

The point is that a possible generalization was obstructed seriously
already at its starting point: if the link $M$ of $(X,0)$ (which is an oriented 3--manifold)
is not a rational homology sphere, then one cannot associate with the link of $g$, $L_g\subset M$,
a Seifert form, and all the linking theory of cycles in $M$, intensively used in the previous cases,
was missing.

There is also  a second  warning. Although in the literature there are a few different
proofs for the semicontinuity property of  {\it hypersurface} singularities (see \cite{St,Var,Var2}),
 in  this general context the semicontinuity  was not even formulated,
and it is not so clear what would be a possible
 Hodge theoretical proof for it. In this general case even the computation of the spectrum
 in concrete examples might create problems, hence we lack even key examples.

The goal of the present note is to surmount all these difficulties, and to propose and prove a
possible semicontinuity inequality. Since in the classical case of hypersurfaces, the
semicontinuity of the spectrum  had several strong applications (mostly as obstructions for deformations),
we expect that the present
results will also find their applications regarding deformations of these more general objects.

The main novelty is the definition of the {\it fractured Seifert form}, defined on the subspace
$\ker(H_1(\S)\to H_1(M))$ of $H_1(\S)$ ($\S$ being the Milnor fiber of $g$).
Moreover, we establish for this new object all the
important properties of the classical Seifert form, and its relation with monodromy and
intersection forms. In this presentation we use intensively the language of
hermitian variation structures of  \cite{Nem2}.

For this fractured Seifert form we can consider the analog of the Levine--Tristram signatures, and we prove
Murasugi type inequality in the presence of a cobordism. Furthermore,
one of the main results shows that the fractured Seifert form determines the Hodge spectrum,
hence, as in the old case,
the Murasugi type inequalities for the signatures
provide semicontinuity properties for the spectrum.

\bigskip

Unless specified otherwise, the homologies usually mean
homologies with rational coefficients.
For a set $A$, the symbol $|A|$ denotes the cardinality. All the manifolds are assumed to be oriented.

\section{Hermitian variation structures -- Generalities}\label{S:HVS}
Hermitian Variation Structures (in short HVS) were introduced in \cite{Nem2}
as a way to encode the `homological Milnor package' of the Milnor fibration. It turns out
that they can be used to connect  knot theory with Hodge theory via the Seifert form of the link.
This approach was exploited  in \cite{BN}: in this language,
the Levine--Tristram signatures for links correspond
to the spectrum of a HVS.

\subsection{Review of hermitian variation structures}\label{s:reviewHVS}
First recall the definition of a HVS. In contrast with e.g. \cite{BN} or \cite{BN2}
we will deal with non-simple  variation structures as well.

\begin{definition}\label{def:HVS} \cite{Nem2}
  For a fixed sign $\ef=\pm 1$,
an \emph{$\ef$--hermitian  variation structure} consist of a quadruple
$(U;b,h,V)$, where
\begin{itemize}
\item[$\bullet$] $U$ is a complex linear space;
\item[$\bullet$] $b\colon U\to U^*$ is an $\ef$--hermitian endomorphism;
\item[$\bullet$] $h\colon U\to U$ is an automorphism preserving $b$;
\item[$\bullet$] $V\colon U^*\to U$ is an endomorphism.
\end{itemize}
These objects should satisfy the following compatibility relations:
\[V\circ\, b=h-I\text{ and }\ol{V}^*=-\ef V\circ\, \ol{h}^*
\ \ \ \ \ \mbox{(`Picard--Lefschetz' formulae)}.\]
Here $\ol{\cdot}$ denotes the complex conjugate and $\cdot^*$ the duality.
\end{definition}
In our applications in the next sections
 we shall always choose the sign $\ef=-1$. Sometimes we refer to the \emph{dimension} of a
 HVS as the dimension of  the underlying linear space $U$.
The prototype of a HVS is provided by a Milnor fibration
of an isolated hypersurface singularity $(\C^{m+1},0)\to (\C,0)$:
$U$ is the middle homology $H_m(F)$ of the fiber,  $b$ the intersection form on it,
$h$ the monodromy, and $V$ is the variation operator, see Section~\ref{ss:clas}.
In this case $\ef=(-1)^m$.

If $V$ is an isomorphism, then we  say that the HVS is \emph{simple}. In such a case $V$ determines
$b$ and $h$ completely by the formulae $h=-\ef V(\ol{V}^*)^{-1}$
and $b=-V^{-1}-\ef {\ol{V}^*}^{-1}$.
If $V$ fails to be an isomorphism, then necessarily $1$ must be an eigenvalue of $h$.

If $b$ is an isomorphism, we say that the HVS is {\it non-degenerate}.
Then $V=(h-I)b^{-1}$, hence the HVS is completely determined by the underlying
\emph{isometric structure} $(U;b,h)$,
see \cite[Remark 2.6a]{Nem2} and \cite{Milnor-forms} for the definition of the isometric structure.

\subsection{Examples of HVS}
Here we shall follow closely \cite{Nem2,BN} (some sign conventions differ from \cite{Nem2}).
For $k\ge 1$, $J_k$ denotes the $(k\times k)$--Jordan block
 with eigenvalue $1$.
\begin{example}\label{ex:mv2k}
For $\lambda\in\mathbb{C}^*\setminus S^1$ and $k\ge 1$, the quadruple
\[
\mv_{\lambda}^{2k}=\left(\mathbb{C}^{2k};%
\left(\begin{matrix}0&I\\\ef I&0\end{matrix}\right),%
\left(\begin{matrix}\lambda J_k&0\\ 0&(1/\bar{\lambda})\cdot {J_k^*}^{-1}\end{matrix}\right),%
\left(\begin{matrix}0&\ef(\lambda J_k-I)\\ (1/\bar{\lambda})\cdot {J_k^*}^{-1}-I&0\end{matrix}\right)\right)
\]
defines a simple and non-degenerate HVS. Moreover, $\mv_{\lambda}^{2k}$ and
$\mv_{1/\bar{\lambda}}^{2k}$ are isomorphic.
\end{example}
\begin{example}\label{ex:mvlambda}
For any $k\ge 1$ there are precisely two non-degenerate $\epsilon$--hermitian forms
 (up to a real positive scaling), denoted by $b^k_\pm$, such that
\[\ol{b}^*=\ef b\text{ and }J_k^*bJ_k=b.\]
By convention, the signs are fixed by $(b^k_{\pm})_{1,k}=\pm i^{-m^2-k+1}$, where $\epsilon=(-1)^m$.
The entries of $b$  satisfy: $b_{i,j}=0$ for $i+j\leq k$ and $b_{i,k+1-i}=(-1)^{i+1}b_{1,k}$.
According to this, for $\lambda\in S^1$,  there are up to an isomorphism two non--degenerate HVS with  $h=\lambda J_k$.
These are
\[\mv^k_{\lambda}(\pm 1)=\left(\mathbb{C}^k;b^k_{\pm},\lambda J_k,(\lambda J_k-I)(b^k_\pm)^{-1}\right).\]
The structures are simple for  $\lambda\not=1$, otherwise not.
\end{example}
\begin{example}\label{ex:mvtilde}
For $k\ge 1$ there are two degenerate simple HVS with $h=J_k$. They are
\[\widetilde{\mv}^k_1(\pm 1)=\left(\mathbb{C}^k;\widetilde{b}_\pm,J_k,\widetilde{V}_{\pm}^k\right),
\ \ \ \mbox{where} \ \ \
\widetilde{b}^k_\pm=\left(\begin{matrix} 0 & 0\\ 0&b^{k-1}_\pm \end{matrix}\right).\]
The entries of $V^{-1}$ satisfy: $(V^{-1})_{i,j}=0$ for $i+j\geq k+2$, $(V^{-1})_{i,k+1-i}=
\pm (-1)^{i+1} i^{-m^2-k}$.
For $k=1$ this is
$\widetilde{\mv}^1_1(\pm 1)=(\mathbb{C},0,I, \pm i^{m^2+1})$.

\end{example}
We use the following uniform notation for the above simple structures:
\begin{equation}\label{eq:wk}
\mw^k_\lambda(\pm 1)=%
\begin{cases}
\mv^k_\lambda(\pm 1)&\text{if $\lambda\in S^1\sm\{1\}$}\\
\widetilde{\mv}^k_1(\pm 1)&\text{if $\lambda=1$.}
\end{cases}
\end{equation}
\subsection{Classification of simple HVS}
In \cite{Nem2} the second author proved that each simple variation structure
is a direct sum of indecomposable ones.
\begin{proposition}\label{p:classification}
A simple HVS is uniquely expressible as a sum of indecomposable ones up to ordering of summands and up to
an isomorphism. The indecomposable pieces are
$\mw^k_\lambda(\pm 1)$ (for $k\ge 1$, $\lambda\in S^1$) and
$\mv^{2k}_\lambda$ (for $k\ge 1$, $0<|\lambda|<1$).
\end{proposition}
\noindent
Hence, for each {\it simple} HVS, say  $\mv$, there exists a collection of numbers
$p^k_\l(u)$ (with    $k\ge 1$, $\l\in S^1$, $u=\pm 1$)
and $q^k_\l$ (with $k\ge 1$ and $0<|\l|<1$) such that
\begin{equation}\label{eq:decomp}
\mv=\bigoplus_{\substack{0<|\lambda|<1\\ k\geqslant 1}} q^k_\lambda\cdot   \mv^{2k}_\lambda
\oplus
\bigoplus_{\substack{|\lambda|=1\\ k\geqslant 1, \ u=\pm 1}} p^k_\lambda(u)\cdot  \mw^{k}_\lambda(u),
\end{equation}
where the symbol $m\cdot\mv$ denotes a sum of $m$ copies of the structure $\mv$.

If a HVS is not simple,
then a direct sum decomposition of the monodromy $h$ (e.g. its Jordan block decomposition)
does not imply the existence of splitting of the whole structure.

\begin{example}[see also \expandafter{\cite[Example 2.7.9]{Nem2}}]\label{ex:nonsplit} Consider
$$b=\bp  0&0&1\\0&0&0\\-1&0&0\ep, \hspace{5mm}
h=\bp  1&0&0\\0&1&1\\ 0&0&1\ep, \hspace{5mm}
V=\bp  0& 1 & 0\\ 1&0&0\\ 0&0&0\ep$$
These matrices define a $(-1)$--HVS, which is indecomposable.
\end{example}
Nevertheless, the next splitting holds
(to see it, write $V$ as a block matrix and use the assumptions).
\begin{lemma}\label{lem:split}
Let $\mv=(U;b,h,V)$ be a HVS. Assume that $U=U_1\oplus U_2$ such that both  $b$ and $h$ have
block decomposition $b=b_1\oplus b_1$ and $h=h_1\oplus h_2$ with
$b_1$ non-degenerate. Then $V=V_1\oplus V_2$
as well, hence  $\mv$ decomposes as a direct sum $\mv_1\oplus \mv_2$ of HVS's.
\end{lemma}

\begin{remark}\label{rem:matrices}
Regarding our identities we use the following matrix notations.

Assume that in the vector spaces $V$ and $W$ we fixed the bases
$\{\xi_i\}_i$, respectively $\{\zeta_j\}_j$. Then a morphism
$A:V\to W$ is represented by the matrix ${\bf A}=\{A_{ij}\}_{ij}$, where
$A(\xi_i)=\sum_jA_{ji}\zeta_j$.
(This means that $(A_{11} \cdots A_{1n})$ is the first line.)
(In other words, if ${\bf v}$ is the column vector  representing $v\in V$
(i.e.  it has entries $\{v_i\}_i$, where $v=\sum_i v_i\xi_i$),
then $A(v)$ is represented by the column vector ${\bf A}\cdot {\bf v}$.)
Any base $\{\xi_i\}_i$ in $V$ determines a dual base $\{\xi_j^*\}_j$ in
$V^*=\Hom_{\R}(V,\R)$ via $\xi_j^*(\xi_i)=\delta_{ji}$,
where $\delta_{ji}$ is the Kronecker symbol.
If $A:V\to V$ is a morphism, then its dual $A^*:V^*\to V^*$ has matrix
representation ${\bf A^*}={\bf A}^t$.
Similarly, if $B:V\times V\to \R$ is a bilinear form, then it is represented
by the matrix ${\bf B}=\{B_{ij}\}_{ij}$, where $B_{ij}:=B(\xi_i,\xi_j)$.
(This means that $B(v,w)$ in matrix calculus is given by ${\bf v}^t\cdot
{\bf B}\cdot {\bf w}$.)
Furthermore, if
$b:V\to V^*$ is defined by $b(v)=B(v,\cdot )$, then the identities
$B_{ik}=B(\xi_i,\xi_k)=b(\xi_i)(\xi_k)=\sum_jb_{ji} \xi_j^*(\xi_k)=b_{ki}$
show that ${\bf B}={\bf b}^t$.
\end{remark}

\subsection{Spectrum and signature of a real simple HVS}\label{s:specHVS}
Let $\mv$ be a simple HVS defined over the real numbers. For simplicity we will
also assume that the coefficients $q^k_\lambda$ in the decomposition \eqref{eq:decomp}
are all zero. In the sequel we define the {\it spectrum} $\Sp$ and the  {\it signature}
of $\mv$. For more details regarding this subsection, see \cite{BN}.
\begin{definition}\label{def:spectrum}
The {\it spectrum}  is a finite set of real numbers from the interval $(0,2]$
with integral multiplicities such that any real
number $\alpha$ occurs in $\Sp$ precisely $s(\alpha)$ times, where
\[s(\alpha)=\sum_{n=1}^\infty\sum_{u=\pm1}\left(\frac{2n-1-u(-1)^{\lfloor \alpha\rfloor}}
{2}p^{2n-1}_\lambda(u)+np^{2n}_\lambda(u)\right), \ \
(e^{2\pi i\alpha}=\lambda).\]
\end{definition}
\begin{definition}\label{def:signatureofHVS}
The \emph{signature} of $\mv$  is the map $\s_{\mv}\colon S^1\sm\{1\}\to\Z$ given by
\[z\mapsto\textrm{signature of}\left((1-z)V+(1-\ol{z})V^t\right).\]
\end{definition}

The spectrum and the signature are related, cf. \cite[Corollary 4.15]{BN}.\footnote{This is Corollary~4.4.9 in the arxiv version of \cite{BN}.}

\begin{lemma}\label{lem:sigandspec}
Let $x\in[0,1]$ be such that $\{x,1+x\}\cap \Sp=\emptyset$. Let $z=e^{2\pi ix}$. Then
\begin{align*}
|\Sp\cap (x,x+1)|&=\frac12(\dim U-\sigma(z))\\
|\Sp\setminus [x,x+1]|&=\frac12(\dim U+\sigma(z)).
\end{align*}
\end{lemma}

This lemma will allow us to use topological methods developed in Section~\ref{S:fracturedd}
(namely, cobordism) to study semicontinuity of the spectrum.

\section{Links in $3$-manifolds}\label{S:fracturedd}
In this section we study links in oriented  closed $3$-manifolds.
Our approach depends on several choices,
for example, choices of the Seifert surface. However, in the applications
in singularity theory, there will
always be a natural choice, dictated by singularity theory.

\subsection{Fractured linking number}\label{ss:fracturedlink}

Let $M$ be a closed connected oriented $3$-manifold. Let $\alpha,\beta\subset M$
be two disjoint one-dimensional cycles. If $M\cong S^3$,
the linking number $\lk(\alpha,\beta)\in\Z$ is a well-defined number.
 We refer to \cite[Chapter V.D]{Rol} for various equivalent definitions.
In this section we extend the definition for arbitrary 3-manifold $M$, but for special 1-cycles.

Assume that
$[\alpha]=[\beta]=0$  in $H_1(M;\Q)$. Then there exists a 2-chain
$A$ such that $\p A=\alpha$.
We denote the algebraic intersection number of $A$ and $\beta$ in $M$ by $A\cdot\beta$ (which
counts intersection points with signs provided that $A$ and $\beta$ are in general position).
If $A$ and $A'$ are two different 2-chains such that $\p A=\p A'=\alpha$, then
$A\cdot\beta=A'\cdot\beta$.
Indeed, the cycle $A\cup -A'$ defines an element in $H_2(M;\Q)$. Then $(A\cup -A')\cdot \beta$
as an intersection product in $M$, is zero, since  $[\beta]=0\in H_1(M)$.
\begin{definition}\label{def:clk}
We define the {\it fractured linking number} $\clk(\alpha,\beta)$ as $A\cdot\beta\in\Q$,
where $\partial A=\alpha$. (By the above discussion  it is independent of the choice of $A$.)
\end{definition}

\begin{remark}
If $M$ is a 3-manifold, one has a linking form on the torsion part of its first homology with values in
$\Q/\Z$. Our construction is different: it assigns an element from $\Q$
to any two disjoint rationally null-homologous cycles. By choosing its  name
`fractured linking number' (instead of
 `linking number') we emphasize the difference and  avoid confusions.
\end{remark}

As the classical linking number, the fractured linking number is symmetric too.

\begin{lemma}\label{lem:fractured}
For any two disjoint null-homologous cycles $\alpha,\beta$ on $M$ we have
 $\clk(\alpha,\beta)=\clk(\beta,\alpha)$.
\end{lemma}
\begin{proof}
It is enough to prove the statement for $\alpha$ and $\beta$ integral cycles.
Let $A$, $B$ be surfaces such that
$\p A=\alpha$, $\p B=\beta$,  and  $A$ and $B$ intersect transversely. Then $A\cap B$ is an
 oriented cobordism between $A\cap\p B$ and $\p A\cap B$. This proof extends to
  the level of chains as well.
\end{proof}

In the classical case, one has another definition of the linking pairing.  Namely,
given two disjoint 1-cycles $\alpha,\beta\subset S^3=\p B^4$,
one takes two 2-chains $A,B$ in the ball $B^4$ such that $\p A=\alpha$ and $\p B=\beta$.
Then $\lk(\alpha,\beta)=A\cdot B$.
We extend this result in a way that we allow an arbitrary four manifold instead of $B^4$, but we need to
impose additional conditions on the chains $A$ and $B$.

\begin{lemma}\label{lem:AcdotB}
Assume that $W$ is a four manifold such that $\p W=M$. Let $\alpha,\beta$ be two disjoint
 null-homologous 1-cycles in $M$. Then for any 2-chains $A,B\subset W$
such that $\p A=\alpha$, $\p B=\beta$ and $[A]=[B]=0\in H_2(W,M;\Q)$ we have
\[\clk(\alpha,\beta)=A\cdot B.\]
\end{lemma}
\begin{proof}
First we show that $A\cdot B$ does not depend on the specific choice of $A$ and $B$.
To this end, assume that we have a 2-chain $A'$ such that $\p A'=\alpha$ and $[A']=0\in H_2(W,M;\Q)$.
Then $A\cup -A'$ is an absolute cycle in $W$ so it represents a class $[A-A']\in H_2(W)$
and $(A\cup -A')\cdot B=[A-A']\cdot [B]$, where the last product is the intersection pairing $H_2(W)\times H_2(W,M)\to\Q$. But this is zero since $[B]=0$. Thus $A\cdot B$
is well-defined.

By picking concrete $A$ and $B$ we will show that $A\cdot B=\clk(\alpha,\beta)$.  To this end,
choose a collar $M\times[0,1)\subset W$, such that $M\times\{0\}$ is identified with
$M=\p W$. Let $C$ be a 2-chain in $M$ such that $\p C=\beta$.
We define $B$ as $C$ with its interior pushed slightly inside $M\times[0,1/2]$.
Clearly
$[B]=0\in H_2(W,M;\Q)$.
Let $A$ be arbitrary 2-chain satisfying the hypothesis of the lemma. We isotope $A$ so that $A\cap (M\times[0,1])=\alpha\times[0,1]$. Then, by construction,
all the intersection points of $A\cap B$ correspond to the intersections of $\alpha$ with $C$.
\end{proof}

\begin{example}
The condition that $A$ and $B$ represent $0\in H_2(W,M)$ is essential, even if $M$ is the 3-sphere.
For example, consider $\C^2$ with coordinates $x,y$
and let $M=\{|x|^2+|y|^2=1\}$, $W_0=\{|x|^2+|y|^2\le 1\}$.
Let $A_0=\{x=0\}$, $B_0=\{y=0\}$ and put $\alpha=A_0\cap M$, $\beta=B_0\cap M$.
Then  $\lk(\alpha,\beta)=1$, as is the algebraic intersection number of $A_0$ with $B_0$.

But now we can define $W$ as $W_0$ blown up in the origin and $A$,
$B$ as the strict transforms of $A_0$ and $B_0$ respectively. Then $A\cap B=\emptyset$,
but still we have $\p A=\alpha$, $\p B=\beta$. Of course, $A$ and $B$ do not represent $0$ in $H_2(W,M)$.
\end{example}

We end up this subsection with an alternative construction of the fractured linking number.
 Let $\alpha$ and $\beta$ be disjoint cycles in $M$, which represent
$0\in H_1(M;\Q)$. Assume that $\alpha$ can be represented by a simple closed loop.
Then $\beta$ defines an element in $H_1(M\sm\alpha;\Q)$,
which is mapped to $0$ by the map $H_1(M\sm\alpha)\to H_1(M)$. We have
\[U:=\ker (H_1(M\sm\alpha;\Q)\to H_1(M;\Q))\cong\Q\]
and there is a canonical choice of
the isomorphism, such that the oriented meridian of $\alpha$ goes to $1$. Then we can
define $\clk(\alpha,\beta)$ to be the
class $[\beta]\in \Q$. We leave it as an exercise to check that the two
 definitions are in fact equivalent.

\subsection{Links and fractured Seifert matrices}

\begin{definition}\label{def:link1}
A \emph{link} $L$ in an oriented connected 3-manifold $M$ is a disjoint
union of embedded oriented circles $K_1\sqcup\dots\sqcup K_n$ in $M$.
A \emph{Seifert surface} of a link $L\subset M$ is a
{\it connected oriented surface} $\S\subset M$ such that $\p\S=L$, and the interior of $\S$
is disjoint from~$L$.
\end{definition}

If such surface exists, we know that $[L]=0\in H_1(M;\Z)$. Conversely,
if $[L]=0\in H_1(M;\Z)$,  the arguments of \cite{Er} or \cite{BNR2} guarantee the existence of $\S$.
However, in the present paper,
all the links that we shall consider will have a Seifert surfaces.

Let $\S$ be a Seifert surface for a link $L$ and $j:\S\hookrightarrow M$ be the inclusion map. We set
\begin{equation}\label{eq:V0}
\vs=\ker\big(j_*\colon H_1(\S;\Q)\to H_1(M;\Q)\big).
\end{equation}
 For any  $\beta\in \vs$ let $\beta^+$ be the cycle $\beta$ pushed slightly
 off $\S$ in the positive normal direction. Obviously, $[\beta^+]=0\in H_1(M;\Q)$.
\begin{definition}\label{def:csp}
The \emph{fractured Seifert pairing} $S\colon \vs\times \vs\to \Q$ for $\S$ is defined by
 $(\alpha,\beta)\mapsto \clk(\alpha,\beta^+)$.
A \emph{fractured Seifert matrix} is a rational square matrix of size $\dim \vs$
such that in some basis of $\vs$ the fractured Seifert pairing
is represented by $S$.
\end{definition}

Usually we shall not make a distinction between a Seifert pairing and a Seifert matrix, see Remark~\ref{rem:matrices}.

In general, $S$ is not $\pm$-symmetric, nevertheless Lemma \ref{lem:fractured} implies the following.
\begin{lemma}\label{lem:disjoint}
If $\alpha,\beta\in \vs$ and $\alpha$ is disjoint from $\beta$, then $S(\alpha,\beta)=\clk(\alpha,\beta)$. In particular, if $\alpha_1,\dots,\alpha_k\in\vs$
are pairwise disjoint, then $S$ restricted to the subspace spanned by $\alpha_1,\dots,\alpha_k$ is symmetric.
\end{lemma}

In this paper we assume that all the links  satisfy the following additional assumptions.
 \begin{definition}\label{def:link}{}
\begin{itemize}
 \item[(a)] A link will be called 0--link if all  component are (rational) null-homologous:
 $[K_j]=0\in H_1(M;\Q)$ for any $j=1,\dots,n$, and if $L$ admits a Seifert surface.

 \item[(b)] A 0--link is called {\it special} if $\clk(K_i,K_j)>0$ for all $i\not=j$.
\end{itemize}
 \end{definition}

 Consider the fractured linking matrix $\{{\mathcal L}(K_i,K_j)\}_{i,j}$
 associated with a {\it special} link $L$.
 Here, ${\mathcal L}(K_i,K_j):=\clk(K_i,K_j)$ for $i\not=j$,
 while ${\mathcal L}(K_i,K_i)$ is determined by the imposed
 identities ${\mathcal L}(\sum_iK_i,K_j)=0$ for any $j$.
 \begin{lemma}\label{lem:nondeg}
 If $L$ is special, then ${\mathcal L}$ is negative semi-definite with
 1-dimensional null space generated by $\sum_iK_i$.
 \end{lemma}
 \begin{proof} (Cf. \cite[Sec. 3]{Neu-inv})
 If $R=\sum_ir_iK_i$, then ${\mathcal L}(R,R)=-\sum_{i<j} (r_i-r_j)^2\clk(K_i,K_j)$.
 \end{proof}
\subsection{Fibred links and monodromy}\label{s:fibred}

Our goal in this section is to study the fractured Seifert matrix associated to a fibred link. Understanding
a decomposition of a fractured Seifert matrix with respect to generalized eigenspaces of the monodromy operator will
lead us to a decomposition of HVS defined for a fibred link.

We begin with the following definition.
\begin{definition}\label{def:hs}
We shall refer to an open book decomposition  $(M,L,p)$
with binding $L$
and projection $p\colon M\sm L\to S^1$ simply as a  \emph{fibred link}.
 We define its  (canonical) \emph{Seifert surface} $\S$ as the page  $p^{-1}(1)$.
 The \emph{monodromy} diffeomorphism (well defined up to an isotopy) will be denoted by
$\hs\colon\S\to\S$.
(Notice that since we consider $L$ to be an oriented link, we require
that $p$ restricted to the oriented meridians $\mu_1,\ldots,\mu_n$ of components of $L$ is an \emph{orientation preserving} diffeomorphism.)
\end{definition}

For any $t\in[0,1]$, set $\S_t=p^{-1}(e^{2\pi it})$ with $\S=\S_0$.
Since $p$ is locally trivial, there exist a smooth family of diffeomorphisms
$\hst\colon\S\stackrel{\cong}{\longrightarrow}\S_t$ induced by trivialization over $[0,t]$,
such that $\hs_0$ is the identity and $\hs_1=\hs$ is the monodromy. These diffeomorphisms are also well defined only
up to isotopy.
Let $h\colon H_*(\S;R)\to H_*(\S;R)$ be the homological monodromy  induced by $\hs$
 for any coefficient ring $R$.

\begin{remark} In the usual definition of the  Seifert matrix, to any cycle $\beta\in H_1(\S)$ we
associate  $\beta^+$ (`the push off of $\beta$ in the positive direction'). In the case of the above fibred
 situation, we set $\beta^+=\hs_{1/2}\beta$. This is common in singularity theory
 too, see e.g. \cite{AGV,Zol}.
\end{remark}

The Wang sequence of the fibration $p\colon M\sm L\to S^1$ is
\begin{equation}\label{eq:wang}
\ldots \to H_1(\S)\to H_1(M\sm L)\stackrel{h-I}{\longrightarrow} H_1(M\sm L)\stackrel{q}{\to} H_0(\S)\to\ldots
\end{equation}
The map $q$ is the  following: a cycle $\alpha\in H_1(M\sm L)$ in general position
with respect to $\S$ is mapped to $(\alpha\cdot\S)$ times the generator of $H_0(\S)\simeq \Z$.
 Since  $q(\mu_i)=1$,  $q$ is onto.

\smallskip
Let $L$ be a special fibred link. Let $j\colon\S\hookrightarrow M$ be the inclusion. We define
\[U_\partial :=\ker\big(
\overline{j_*}:H_1(\S)/{\rm im}(h-I)\to H_1(M)\big),\]
where $\overline{j_*}$ is induced by  $j_*$. Later on we shall define a lift of $\vb$ to
a subspace of $H_1(\S)$.

Let $\mu_i$ be the oriented meridian of $K_i$ in $M$. We have the following commutative diagram,
\begin{equation}\label{eq:newdiagram}
\xymatrix{%
&&0\ar[d]&&\\
&&\Q\langle\mu_i\rangle_{i=1}^n\ar^m[d]\ar^{\ol{m}}[dr]&&\\
0\ar[r]&\vb\ar[r]\ar^{\ol{j}_*}[dr]&H_1(M\sm L)\ar[r]\ar^r[d]&H_0(\S)\cong\Q\ar[r]&0\\
&&H_1(M)\ar[d]&&\\
&&0&&
}
\end{equation}
where the horizontal exact sequence is induced by (\ref{eq:wang}).
The map $m$ is induced by inclusions of meridians,
the map $r$ is induced by inclusion. The diagonal maps are compositions.

\begin{lemma}\label{lem:new}\
\begin{itemize}
\item[(a)] The vertical line of \eqref{eq:newdiagram} is a short exact sequence.
\item[(b)] The map $\ol{m}$ is given by $\sum q_i\mu_i\mapsto\sum q_i\in H_0(\S)$, hence it is onto.
\item[(c)] The map $\ol{j}_*$ is onto. Furthermore $\dim\vb=n-1$ and $\vb\cong\{\sum q_i\mu_i \colon \sum q_i=0\}$.
\end{itemize}
\end{lemma}
 \begin{proof} {}
\begin{itemize}
\item[(a)] Use $H_2(M, M\setminus L)=\Q\langle \mu_i\rangle _{i=1}^n$ and the fact that
 $H_2(M)\to H_2(M,M\setminus L)$ is trivial due to the fact that each $[K_i]=0$.

\item[(b)] is clear, the map $\ol{m}$ sends a meridian to $M\sm L$ and $q$ sends it further to $1$, for each meridian intersects $\S$ precisely once.

\item[(c)] surjectivity of $\ol{j}_*$ follows from (b) and  diagram chasing. The rest is clear.
\end{itemize}
 \end{proof}

The intersection form $\alpha \cdot_\S\beta $ on $H_1(\S)$ has the following compatibility properties.
(Part (b) is the analogue of a Picard--Lefschetz formula from singularity theory.)

\begin{lemma}\label{lem:new2} \
Assume that $L$ is a special fibred link and
  $\alpha=(h-I)\gamma$ for some $\gamma\in H_1(\S;\Q)$.
Then $\alpha\in U^\Sigma$. Moreover, the following hold.
\begin{itemize}

\item[(a)] If $\beta\in\ker(h-I)$ then $\alpha\cdot_\S\beta=0$.

\item[(b)] If $\beta\in \vs$ then
$\clk(\alpha,\beta^+)= \gamma\cdot_\S\beta.$

\item[(c)] Denote the homology classes in $H_1(\S)$ determined by the boundary components by $\{K_i\}_{i=1}^n$
(subject to the single relation $\sum_i K_i=0$). Let ${\mathcal K}$ be the subspace of $H_1(\S)$ generated by these components.
Then ${\mathcal K}\cap \im(h-I)=0$, hence ${\mathcal K}$ injects to  $H_1(\S)/{\rm im}(h-I)$
with image exactly $U_\partial$.
\end{itemize}
\end{lemma}
\begin{proof} The first statement follows from Wang exact sequence, which shows that the class of
$\alpha$ is zero already in $H_1(M\setminus L)$.

\begin{itemize}
\item[(a)] $\gamma\cdot_\S\beta=h\gamma\cdot_\S
h\beta=(\gamma+\alpha)\cdot_\S\beta=\gamma\cdot_\S\beta+\alpha\cdot_\S\beta$, since
$\cdot_\S$ is $h$-invariant.

\item[(b)]
Consider $A=\bigcup_{t\in[0,1]}\hst \gamma$. The boundary of $A$ is
$h\gamma-\gamma$, which is homologous to $\alpha$.
Hence $\clk(\alpha,\beta^+)=A\cdot\beta^+$. We can assume that $A$
is in general position with respect to $\beta^+=h^\S_{1/2}\beta$.
\[ A\cdot\beta^+=(A\cap\S_{1/2})\cdot_{\S_{1/2}}\beta^+=
\hs_{1/2}\gamma\cdot_{\S_{1/2}}\hs_{1/2}\beta=\gamma\cdot_\S\beta.\]

\item[(c)] Set $R=\sum_ir_iK_i=(h-I)\gamma$. Then by part (b) $\clk(R,K_j)=\gamma\cdot_\S K_j=0$. Hence, by
Lemma \ref{lem:nondeg}, $R=0$ in $H_1(\S)$. Since each $K_j$ is zero-homologous, ${\mathcal K}$
 injects in $U_\partial$. Since these spaces  have the same dimension, they are isomorphic.
\end{itemize}
\end{proof}

\begin{convention}
In the sequel
we will not make distinction between ${\mathcal K}$ and $U_\partial$, we think about
$U_\partial$  as the
{\it kernel of the intersection pairing on $H_1(\S)$}, that is $\vb$ is regarded as a subspace of $\vs$.
\end{convention}

Consider the generalized eigenspace decomposition of $h$
\begin{equation}\label{eq:firstdecomp}
H_1(\S;\Q)=\vn\oplus\ve,
\end{equation}
corresponding to eigenvalue $\neq 1$, respectively $=1$.
 Both subspaces $\vn$ and $\ve$ are monodromy invariant and orthogonal with respect to $\ \cdot_\S\ $.

\begin{corollary}\label{cor:fromwang}\
\begin{itemize}
\item[(a)] The subspace $\vn$ belongs to $\vs$. Furthermore $\ve\cap \vs=\vi\oplus\vb$, where $\vi=\ve\cap\im (h-I)$.

\item[(b)] The monodromy $h$ preserves the direct sum decomposition
\begin{equation}\label{eq:v0-decomp}
\vs=\vn\oplus\vi\oplus\vb.
\end{equation}
\item[(c)] The components in (\ref{eq:v0-decomp}) are
pairwise orthogonal with respect to the intersection form $\ \cdot_\S \ $.
\item[(d)] The restriction of the
intersection form on  $\vn$ is non-degenerate.
\end{itemize}
\end{corollary}
\begin{proof}
For (a) use Wang exact sequence, Lemma \ref{lem:new} and part (c) of Lemma \ref{lem:new2}. Part
(b) is clear. Next, we prove (c).
The kernel of the intersection form, $\vb$
 is orthogonal to everything. Next,  $\vi\perp \vn$, since  $\ve\perp \vn$. For (d), take
$\alpha$ from the kernel of $\cdot_\S|\vn$. Then $\alpha\in \ker(\cdot_\S)$, hence $\alpha\in U_\partial$. But
$U_\partial\cap \vn=0$.
\end{proof}

\begin{proposition}\label{prop:sisinv}
For the fractured Seifert pairing $S$ the following facts hold.

\begin{itemize}
\item[(a)] $S$
 is monodromy invariant, i.e. $S(h\alpha,h\beta)=S(\alpha,\beta)$ for any $\alpha,\beta\in \vs$,

\item[(b)] $S$ satisfies $S(\alpha,\beta)=S(h\beta,\alpha)$,

\item[(c)] $S$ satisfies $S(\alpha,\beta)=S(\beta,\alpha)=0$ for  $\alpha\in\vn$ and $\beta\in\vi\oplus\vb$, and

\item[(d)] $S$ has block structure with respect to the decomposition~\eqref{eq:v0-decomp}.
\end{itemize}
\end{proposition}
\begin{proof}\
\begin{itemize}
\item[(a)] The monodromy $\hs$ extends to an automorphism of $M$,
 still denoted by the same $\hs$. We clearly have $\clk(\alpha,\beta^+)=\clk(\hs\alpha,(\hs\beta)^+)$.
 Hence $S(\alpha,\beta)=S(h\alpha,h\beta)$.

\item[(b)] We have $S(\alpha,\beta)=\clk(\alpha,\beta^+)=\clk(\alpha^+,\hs\beta)=\clk(\hs\beta, \alpha^+)=
S(h\beta,\alpha)$.

\item[(c)] Take $\gamma\in \vn$ such that $(h-I)\gamma=\alpha$. Then $S(\alpha,\beta)=\gamma\cdot_\S\beta=0$ by
Corollary \ref{cor:fromwang}.
Similarly, $S(\beta,\alpha)=S(h\alpha,\beta)=S(\alpha,h^{-1}\beta)=\gamma\cdot_\S h^{-1}\beta=0$.

\item[(d)]
By part  (c) it is enough to ensure that $S$ has block structure on $\vi\oplus\vb$. Let $\alpha
=(h-I)\gamma\in\vi$ and $\beta\in\vb$. Then $S(\alpha,\beta)=\gamma\cdot_\S\beta=0$.
\end{itemize}
\end{proof}

Corresponding to part (d) of Proposition~\ref{prop:sisinv},
we write $\Sn$, $\Si$ and $\Sb$ for the fractured Seifert pairing
 restricted to $\vn$, $\vi$ and $\vb$ respectively.

\subsection{Non-degeneracy of $S$. Simple fibred links}
In this subsection we give sufficient conditions for the
fractured Seifert matrix of a fibred link to be non-degenerate.

\begin{proposition}\label{prop:Si-nondeg} \

\begin{itemize}
\item[(a)] $\Sn$ and $\Sb$  are  non-degenerate.

\item[(b)] If $\vi \subset \ker(h-I)$, then $\Si$ is non-degenerate as well.
\end{itemize}
\end{proposition}

\begin{proof}
We begin with (a). Since $h-I$ is invertible on $\vn$ and the intersection pairing on
$\vn$ is non-degenerate, the statement follows from Lemma~\ref{lem:new2}(b).
The non-degeneracy of $\Sb$ is built in our assumptions: Lemma~\ref{lem:disjoint} together with Definition \ref{def:link} and
Lemma \ref{lem:nondeg} show that $\Sb$ is actually a negative definite symmetric matrix. We note, that in the statement of Lemma~\ref{lem:nondeg},
we use the word `negative semi-definite', but on $\vb$ this 1-dimensional null--space of the fractured linking form is killed by
Lemma~\ref{lem:new}(c).

We continue with (b).
Recall  that by Lemma~\ref{lem:new2} we have $\vb \cap\vi=0$. Set $\ol{\ve}:=\ve/\vb$.
Then $\vi$ projects isomorphically to a subspace of $\overline{\ve}$.
Let $\alpha_1,\dots,\alpha_k$ be linearly independent elements in $\vi$, such that their representatives form a basis in $\vi/\vb$.
Let us choose $\beta_1,\dots,\beta_k$
in $\ve$ such that $(h-I)\beta_j=\alpha_j$.
Let $\vbb$ be the space spanned by $\beta_1,\dots,\beta_k$. Clearly $\vbb\to\vbb/\vi$ is an isomorphism.
Finally, take a subspace $\vf$
of $\ve$ such that
\begin{equation}\label{eq:vfix}
\ve=\vf\oplus\vi\oplus\vb\oplus\vbb,
\end{equation}
and such that $\vf$ is $h$--invariant.
This is possible because the assumption guarantees that there are no Jordan block of size 3 or larger and $\vbb\oplus\vi$ corresponds
to Jordan blocks of size $2$.

We have $\vi\oplus \vf\subset \ker(h-I)$, so by Lemma \ref{lem:new2}(a)
$\vi\perp (\vi\oplus \vf)$. On the other hand, the induced intersection form on $\vi\oplus\vf\oplus\vbb$ is non-degenerate, this follows
from the fact that the kernel of the intersection form on $\ve$ is exactly $\vb$, see Corollary~\ref{cor:fromwang}(c).
We conclude that
the block matrix
$\{\alpha_i\cdot_\S\beta_j\}_{i,j}$ should be non-degenerate.
But by Lemma~\ref{lem:new2}(b) this  is the matrix of $\Si$
in the basis $\alpha_1,\dots,\alpha_k$.
\end{proof}

\begin{remark}
The subspaces $\vbb$ and $\vf$ of $\ve$ that were defined in the above proof, will be important in Section~\ref{s:vf}. Their definition
depends on various choices, in Section~\ref{s:vf} we will put additional conditions on $\vbb$ and $\vf$.
\end{remark}

The assumption that $\vi\subset \ker(h-I)$
is equivalent to the absence of Jordan blocks of size $3$ or larger of eigenvalue $1$
in the Jordan block decomposition of $h$. This always holds if $L$ is a
graph link in a graph manifold (see e.g. \cite{EN}). More generally,
if in the Nielsen--Thurston decomposition of $h$ there are no pseudo--Anosov pieces, then $h$ cannot
have Jordan blocks of size larger than $2$ (with whatever eigenvalue)
by \cite[Corollary 13.3]{FM}.

\begin{definition}\label{def:simple}
A special fibred link $(M,L,p)$ is called \emph{simple} if
the monodromy has no Jordan block of size more than 2 with eigenvalue $1$.
\end{definition}

\begin{corollary}\label{cor:simple}
The fractured Seifert matrix of a simple fibred link is non-degenerate.
\end{corollary}

\subsection{Complementary space to $\vs$}\label{s:vf}
Assume that $L$ is a simple fibred link.
Let us consider the space $\vf $ in the proof of Proposition \ref{prop:Si-nondeg}.
Since the intersection form on $\vi\times \vbb$ is non-degenerate
we can choose $\vf$ (by adding vectors from $\vi$ to base elements of the original $\vf$)
such that $\vbb\perp \vf$ too and $\vf$ still remains  $h$-invariant.
Decomposition \eqref{eq:vfix} has the property that
the space $\vf$ is orthogonal to $\vi\oplus\vb\oplus\vbb$, $\vi\perp\vi$,
and $\ker(h-I)=\vf\oplus \vi\oplus \vb$.
On $\vf\oplus \vi\oplus \vbb$ the intersection from is non-degenerate. Then automatically one also has the following result.

\begin{corollary}\label{cor:vfnondeg}
The intersection form of $H_1(\S)$ restricted to $\vf\times\vf$ is non-degenerate, and
$\vf$ is even dimensional.
\end{corollary}

Decomposition~\ref{eq:vfix} together with Corollary~\ref{cor:fromwang} gives also
\begin{equation}\label{eq:complement}
H_1(\S)=\vs\oplus\vbb\oplus\vf.
\end{equation}

To conclude this section let us write a corollary to Lemma~\ref{lem:new}(c).

\begin{corollary}\label{cor:dimvs}
For any special fibred link (not necessary simple) we have
\[\dim\vs=\dim H_1(\S)-\dim H_1(M).\]
\end{corollary}
\begin{proof}
Lemma~\ref{lem:new}(c) tells that $\ol{j}_*$ is onto, hence $H_1(\S)\to H_1(M)$ is onto, and $\vs$ is defined as the kernel of this map.
\end{proof}

\section{Cobordisms and signatures}
\subsection{Cobordisms of links}

We begin with the following definition.
\begin{definition} Fix two links $L_0\subset M_0$ and $L_1\subset M_1$ as in \ref{def:link1}.
A \emph{cobordism of links} connecting $(M_0,L_0)$ and $(M_1,L_1)$
is a pair $(W,Y)$, where $W$ is a 4-manifold and $Y\subset W$ is a surface, such that $\p W=-M_0\sqcup M_1$,
$\p Y=-L_0\sqcup L_1$.

A cobordism will be called \emph{standard} if $W\cong M_0\times[0,1]$.
\end{definition}

Compatibility with Seifert surfaces of links is formulated as follows.
\begin{definition}\label{def:Seifertcobordism}
Let $(M_0,L_0)$ and $(M_1,L_1)$ be two links and $\S_0,\S_1$ Seifert surfaces respectively for $(M_0,L_0)$ and $(M_1,L_1)$.
A pair $(W,\O)$ is a \emph{Seifert cobordism} between links $(M_0,L_0)$ and $(M_1,L_1)$
and their Seifert surfaces $\S_0$ and $\S_1$ if
\begin{itemize}
\item[$\bullet$] $W$ is a 4-manifold with boundary $-M_0\sqcup M_1$;
\item[$\bullet$] $\O$ is a 3-manifold with corners;
\item[$\bullet$] $\S_0:=\O\cap M_0$ and $\S_1:=\O\cap M_1$ are the Seifert surfaces for $L_0=\p\S_0$ and $L_1=\p\S_1$;
\item[$\bullet$] $\p\O=\S_0\cup Y\cup\S_1$, where $Y\subset W$ and $\p Y=-L_0\cup L_1$.
\end{itemize}
\end{definition}
\begin{remark}\label{rem:iswsmooth}
The condition that $W$ is a manifold might be slightly relaxed. More precisely,
we can allow $W$ to be singular away from $\p W\cup\O$. For example, we can assume that
there exists a finite number
of points $w_1,\dots,w_s\in W\sm(\partial W\cup \O)$
such that $W\sm\{w_1,\dots,w_s\}$ is a smooth manifold.
In fact, in the applications we do not use
the smoothness of $W$. On the other hand,  the smoothness of $\O$ is exploited by
its Poincar\'e duality.
\end{remark}

We wish to study, how  the fractured Seifert matrices does change under the Seifert cobordism.
The situation is standard in the classical case.
To simplify the notation,  we
shall first restrict ourselves to the case when $M_1$, $L_1$ and $\S_1$ are empty.
The manifolds $M_0$, $L_0$ and $\S_0$ will be denoted by $M$, $L$, $\S$.
The inclusions $\S\hookrightarrow\O$,   $\S\hookrightarrow M$ and $(\O,\S)\hookrightarrow(W,M)$ will be denoted by $i$,  $j$, resp. $k$.

Let $\vnu\subset \vs$ be the space of those elements  $\alpha\in\ker i_*$ for which there exists
 $A\in H_2(\O,\S;\Q)$
such that $\p A=\alpha$ and $k_* A=0$ (see diagram (\ref{eq:funddiagram})).

\begin{proposition}\label{prop:thenullspace}
Let $\alpha,\beta\in \vnu$.
Then, $S(\alpha,\beta)=0$.
\end{proposition}
\begin{proof}
Let $B^+$ be the cycle $B$ pushed off $\O$ in the positive normal direction.
Obviously $\p B^+=\beta^+$ and $B^+$ is a zero element in $H_2(W,M;\Q)$.
By Lemma~\ref{lem:AcdotB} we have $\clk(\alpha,\beta^+)=A\cdot B^+$.
But $A\subset\O$ and $B^+$ is disjoint from $\O$.
\end{proof}

Next we search for a bound for $\dim \vnu$. We will need the following diagram
\begin{equation}\label{eq:funddiagram}
\xymatrix{
\ & H_2(\O,\S)\ar[r]^{\p}\ar[d]_{k_*} &
H_1(\S)\ar[r]^{i_*}\ar[d]_{j_*} & H_1(\O)\ar[d]\\
0\to\coker(H_2(M;\Q)\to H_2(W;\Q))\ar[r]^(0.7){l_*} & H_2(W,M)\ar[r]^{\p'} & H_1(M)\ar[r]^{a_*} & H_1(W),
}
\end{equation}
(where the rows are long exact sequences of pairs) and the next terminology as well.
\begin{definition}
The \emph{irregularity} of the Seifert cobordism $(W,\O)$ is
\begin{equation}\label{eq:irr}
\Irr_2:=\dim\coker\big(H_2(M;\Q)\to H_2(W;\Q)\big).
\end{equation}
(The subscript $2$ suggests that the irregularity is related to the map on the second homologies.
Later, in a more specific case, we shall also introduce $\Irr_1$.)
\end{definition}
We have the following estimates.
\begin{lemma}
\begin{align*}
(a) \hspace{1cm} & \dim \vnu\ge \dim(\ker i_*\cap \ker j_*)-\Irr_2\\
(b) \hspace{1cm} & \dim(\ker i_*\cap\ker j_*)\ge\dim\ker i_*-\dim\ker a_*.
\end{align*}
\end{lemma}
\begin{proof}
(a) Since $\vnu=\p(k_*^{-1}(0))$ and $\ker i_*\cap\ker j_*=\p(k_*^{-1}(\ker \p'))$, one has
\[\dim (\ker i_*\cap \ker j_*)/\vnu\leq \dim(k_*^{-1}(\ker \p')/k_*^{-1}(0)
=\dim \ker \p'.\]

(b) $0\to \ker i_*\cap\ker j_*\to \ker i_*\stackrel{j_*}{\longrightarrow} \ker a_*$ is exact.
\end{proof}

We remark that $\dim\ker a_*$ does not depend on the cobordism of the link itself
(that is, on $L,\ \S,\ \Omega$), but only on $M$ and $W$. Finally, we estimate $\dim\ker i_*$.

\begin{lemma}\label{lem:poinc}
One has  the following estimate
\[\dim\ker i_*\ge b_1(\S)-\frac12b_1(\S\cup Y).\]
\end{lemma}
\begin{proof}
Decompose  $i_*$  as $H_1(\S)\stackrel{b_*}\longrightarrow H_1(\S\cup Y)
\stackrel{c_*}{\longrightarrow}H_1(\O)$.
 Then
\[\dim\ker i_*-\dim\ker b_*=\dim (\im b_*\cap \ker c_*)\ge
\dim \im b_*+\dim\ker c_*-b_1(\S\cup Y).\]
Clearly, $\dim\ker b_*+\dim (\im b_*)=b_1(\S)$. As $\S\cup Y=\p\O$ and
$\O$ has dimension three, by the Poincar\'e duality arguments one gets
$\dim\ker c_*= b_1(\S\cup Y)/2$.
\end{proof}

Now we combine our last three statements.

\begin{corollary}\label{cor:bound}
If $(W,\O)$ is a Seifert cobordism such that $(M_1,L_1,\S_1)=\emptyset$, then $\vs$ contains the
 subspace $\vnu$ of dimension at least
\[ b_1(\S)-\frac12 b_1(Y\cup\S)-\Irr_2-\dim\ker a_*\]
on which the Seifert form identically vanishes.
\end{corollary}

We will use Corollary~\ref{cor:bound} to control the fractured signatures, which we now define.

\subsection{Fractured signatures}

\begin{definition}\label{def:fracturedsig}
Let $(M,L)$ be a special link,  $\S$ its Seifert surface, and  $S$ the
fractured Seifert pairing on $\vs$. The \emph{fractured signature}
is a function $\s\colon S^1\sm\{1\}\to\Z$ defined as
\[z\mapsto \textrm{signature}((1-z)S+(1-\ol{z})S^t).\]
The \emph{fractured nullity} is a function $n \colon S^1\sm\{1\}\to\Z$ defined as
\[z\mapsto \dim\ker ((1-z)S+(1-\ol{z})S^t).\]
\end{definition}

Here $\cdot^t$ denotes the transpose (in the matrix notation), or simply $S^t(\alpha,\beta)=S(\beta,\alpha)$.

\begin{remark}
If $S$ is non-degenerate then $n(z)=0$ for almost all values of $z$.
Furthermore, if $(M,L)$ is fibred and $\S$ is a fiber, then
$\s$ is a piecewise constant function with possible jumps only at roots
of the characteristic polynomial of the monodromy
$\Delta(z):=\det(h-I\cdot z)$.
Indeed, by Proposition~\ref{prop:sisinv}(b) one has the identity $(1-z)S+(1-\ol{z})S^t=(1-\ol{z})S(h^{-1}-I\cdot z)$.
\end{remark}

Fractured signatures are motivated by cobordisms: the next Theorem
\ref{thm:signatures} is at the core of the
present article. To prove it, we start with the following standard fact from linear algebra.

\begin{lemma}\label{lem:lagrang}
If $\vnu\subset \vs$ is a subspace such that for $\alpha,\beta\in \vnu$
we have $S(\alpha,\beta)=0$, then for all $z\in S^1\sm\{1\}$:
\[|\s(z)|\le \dim \vs-2\dim \vnu+n(z).\]
\end{lemma}

The next main result is the starting point in proving
our semicontinuity results.

\begin{theorem}\label{thm:signatures}
Let $(W,\O)$ be a Seifert cobordism of links $(M_0,L_0)$ and $(M_1,L_1)$.
Then for all $z\in S^1\sm\{1\}$ we have
\begin{multline*}
|\s_0(z)-\s_1(z)|\le \dim \vs_0+\dim \vs_1-2b_1(\S_0)-2b_1(\S_1)+
b_1(Y\cup\S_0\cup\S_1)+2\Irr_2+\\+2\dim\ker(H_1(M_0\cup M_1)\to H_1(W))+n_0(z)+n_1(z),
\end{multline*}
where $\vs_j=\ker(H_1(\S_j)\to H_1(M_j))$ and  $\S_j=\O\cap M_j$ is the Seifert surface of $L_j$ ($j=0,1$).
\end{theorem}
\begin{proof}
We define $(M,L)=(M_1\cup -M_0,L_1\cup -L_0)$. Let also $\S=\S_1\cup-\S_0$.
Then the fractured signatures of $L$ are $\s(z)=\s_0(z)-\s_1(z)$ and
the nullities satisfy $n(z)=n_0(z)+n_1(z)$. Indeed, if $S_0$ and $S_1$ are
fractured Seifert matrices for $(M_0,L_0)$ and $(M_1,L_1)$, respectively,
then $S_1\oplus -S_0$ is a fractured Seifert matrix for $S$.
The theorem is now a direct consequence of Corollary~\ref{cor:bound} and Lemma~\ref{lem:lagrang}.
\end{proof}

\section{Hermitian variation structures associated with a simple fibred link}\label{s:chvs}

In this section we will associate with a simple fibred link two HVSs. One of them is given
by the fractured Seifert matrix, the other by the classical variation map.
We prove that they determine each other.

\subsection{The `fractured' and `mended' HVS for simple fibred links}\label{ss:chvs}

Let $(M,L,p)$ be a simple fibred link. Let $\S$ be its
Seifert surface, $h\colon H_1(\S)\to H_1(\S)$ the monodromy, $b$ the hermitian intersection form
on $H_1(\S)$,
and  $S\colon \vs\times \vs\to\C$ the fractured Seifert pairing.

\begin{definition}\label{def:fracturedHVS}
The \emph{fractured
hermitian  variation structure} associated to the simple fibred link $L$ is the structure
$\mvc=(\vs;b|_{\vs},h|_{\vs},(S^{t})^{-1})$ (defined already over $\Q$).
\end{definition}
By
Lemma~\ref{lem:new2}(b) and Proposition~\ref{prop:sisinv}(b) (and using the last line of
Remark \ref{rem:matrices} too) one checks that the above system forms a HVS.
Since $(S^t)^{-1}$ is invertible, cf. Proposition~\ref{prop:Si-nondeg},  $\mvc$ is simple.

According to the direct sum $\vs:=\vn\oplus\vb\oplus\vi$
 (cf. Section~\ref{s:fibred})  the fractured HVS
decomposes into  a direct sum $\mvc:=\mvn\oplus\mvb\oplus\mvi$ of HVS
as well.  This follows from  the Splitting Lemma~\ref{lem:split}, since
the pair $(b|_{\vs},h|_{\vs})$ admits a direct sum decomposition, and
$b|_{\vs}$ is non-degenerate on $\mvn\oplus\mvi$, cf. Corollary \ref{cor:fromwang}.
The components are the following.

\begin{itemize}
\item The quadruple $\mvn=(\vn;\bn,h|_{\vn},(\Sn^t)^{-1})$ (the natural restrictions on $\vn$);

\item On $\vb$ the $h|_{\vb}$ is trivial and  $b|_{\vb}= 0$.
Hence  $\mvb=(\vb;0,I,(\Sb^t)^{-1})$. By Lemma
\ref{lem:nondeg} $(\Sb^t)^{-1}$ is negative definite, thus $\mvb=(n-1)\cdot \mw^1_1(+1)$.

\item
On $\vi$
the form $b|_{\vi}$ is $0$ (by Lemma~\ref{lem:new2}(a)) and  $h|_{\vi}$ is the
identity (since $L$ is simple). Thus $\mvi=(\vi;0,I,(\Si^t)^{-1})$.  In particular, $\vi$ is a union of copies of $\mw^1_1(+1)$ and $\mw^1_1(-1)$.
\end{itemize}

Although the operators $b$ and $h$  of $\mvc$ can be extended to
(the intersection form and monodromy on)   $H_1(\S,\C)$, the extension of $(S^t)^{-1}$ is not immediate.

Nevertheless, we wish to define such an extension, however the extension
will  fail to be simple.
 First, we introduce a HVS on $\vf$ (see Section~\ref{s:vf}).

\begin{definition}\label{def:fixext}
The structure $\mvf$ on $\vf$ is the HVS structure determined by  non-degenerate
isometric structure $(\vf;b|_{\vf},h|_{\vf})$ (cf. Corollary~\ref{cor:vfnondeg}).
\end{definition}

Note that $h$ is the identity, $V=0$, and $b|_{\vf}$ is fully antisymmetric. Hence
$\mvf$ is a direct sum of $\frac12\dim\vf$ copies of $\mv^1_1(+1)\oplus \mv^1_1(-1)$.

\begin{definition}\label{def:mended}
The \emph{mended} HVS  associated with a simple fibred link, denoted by $\mvm$,  is a direct sum
$\mvm:=\mvb\oplus\mvn\oplus\mv'\oplus\mvf$,
where $\mv'$ is constructed  from $\mvi$ by replacing each summand
 $\mw^1_1(\pm 1)$ of  $\mvi$ by a copy of $\mv^2_1(\pm 1)$.
\end{definition}

It is straightforward that the operator $h$ of $\mvm$ is the monodromy
$ H_1(\S)\to H_1(\S)$, the form $b$ of $\mvm$
is the intersection form of $\S$, and the restriction of $\mvm$ on $\vs$ is $\mvc$.
 The operator $V$ of $\mvm$ is zero on $\vf\oplus U_B$, which is exactly the kernel of $V$.
This space is isomorphic to $H_1(M;\C)$,
since $j_*:H_1(\S)\to H_1(M)$ is onto, cf. Lemma \ref{lem:new}(c).

The form of the above extensions is motivated by Theorem \ref{th:mendfib}.

\subsection{The `classical' homological variation structure on $H_1(\S)$}\label{ss:clas} \
Fix a fibred link $(M,L,p)$ and consider the fibration $p:M\setminus L\to S^1$ with fiber $\Sigma=p^{-1}(1)$.
Let $\hs:\S\to\S$ be the geometric monodromy as in Definition~\ref{def:hs}. We can assume
that $\hs|_{\partial \S} $ is the identity (this is how usually we recover
$M$ from $\hs$).

The {\it variation map} $\Var:H_1(\S,\partial\S)\to H_1(\S)$ is given by
$[x]\mapsto [\hs(x)-x]$ for any relative cycle $x$ of $(\S,\partial \S)$, see e.g.~\cite{Lo}.
After identifying  $H_1(\S,\partial\S)$ with the  dual of $H_1(\S)$ we obtain the
{\it homological (real) VHS associated with the fibration $p$},
$\mv_{{\rm fib}}:=(H_1(\S); b,h, \Var)$,
where $b$ and $h$ are the intersection form and the algebraic monodromy.
($\mv_{{\rm fib}}$ is defined already over $\Z$.)
\smallskip

It is known that when $M$ is a $\Q$-- or $\Z$--homology sphere, then
$\Var= (S^t)^{-1}$ (over $\Q$ or $\Z$ respectively). This fact is generalized in the next statement.

\begin{theorem}\label{th:mendfib} $\mv_{{\rm fib}}=\mvm$. Moreover,
$\mv_{{\rm fib}}$ determines the fractured Seifert
 form $S$, while $S$ and the integer $n$ (the number of components of $L$)
 determine $\mv_{{\rm fib}}$ as well.
\end{theorem}
\begin{proof}
Let us write $\widetilde{U}:= \vi\oplus U_B$. Note that the monodromy $h$ and the intersection form
$b$ on $U=H_1(\S)$ have block decomposition according to the direct sum decomposition
$U=U_{\not=1}\oplus \vf\oplus \widetilde{U}\oplus \vb$. Since the kernel of $b$ is exactly
$\vb$, by Lemma \ref{lem:split} the whole $\mv_{{\rm fib}}$ splits into
$\mv_{\not=1}\oplus \mv_{fix}\oplus \widetilde{\mv}\oplus \mvb$.

Consider now the homological exact sequence of the pair $(M,\S)$. In the presence of the
open book decomposition $p$, the boundary operator $\partial: H_2(M,\S)\to H_1(\S)$ can be identified with
$\Var:H_1(\S,\partial \S)\to H_1(\S)$. Hence, one has the exact sequence (see also \cite[(2.6)(b)]{Ste})
\begin{equation}\label{ex:VAR}
0\to H_2(M)\to H_2(\S,\partial \S)\stackrel{\Var}{\longrightarrow} H_1(\S)\to H_1(M)\to 0.
\end{equation}
This has several consequences. First, $\im \Var=U^\S$. Hence, if $\alpha\in H_1(\S,\partial \S)$ and
$\beta\in  U^\S$, and $\alpha(\beta)$ denotes the pairing (duality)
$H_1(\S,\partial\S)\otimes  H_1(\S)=
U^*\otimes U\to \R$, then $S(\Var\alpha,\beta)$ is well-defined, and, in fact, it equals
$\alpha(\beta)$ (whose proof is analogous to the proof of \ref{lem:new2}(b)). In matrix notations,
$  \alpha(\beta)=S^t(\Var\alpha)(\beta)$, hence $S^t\cdot \Var$ is the
identity whenever $S$ is well-defined.
In particular, $\Var$ extends $(S^t)^{-1}$ from $U^\S$ to $U$.

The extension is special: by (\ref{ex:VAR}) the rank of the kernel of $\Var$ is $\dim H_1(M)$ which equals
$\dim (\vf\oplus U_B)$ (cf. the end of \ref{ss:chvs}), the complementary space of $U^\S$ in $U$.

Since $b|_{\vf}$ is non-degenerate on $\vf$, it determines  the HVS, hence this component agrees with the
extension  \ref{def:fixext} of $\mvm$. Finally, the extension from
$\vi$ to $\widetilde{U}$ with the imposed kernel
property mentioned above, imposes the modification
 $\mw^1_1(\pm 1) \,\mapsto \, \mv^2_1(\pm 1)$ from \ref{def:mended}.
 This ends the proof of  the identity $\mv_{{\rm fib}}=\mvm$.

Let us recall the type of the blocks of $(\mvm)_{=1}$.
$\mvf$ is a direct sum of $\frac12\dim\vf$ copies of $\mv^1_1(+1)\oplus \mv^1_1(-1)$,
 $\mvb=(n-1)\cdot \mw^1_1(+1)$, while $\widetilde{\mv}$ has (say)
  $c_{\pm}$ copies of $\mv^2_1(\pm 1)$.

 Since all these types are different, it is easy to delete the extended part:
 $\mvf$ is deleted, $\mvb$
 is preserved, while $\oplus_\pm\,
 c_{\pm}\cdot \mv^2_1(\pm 1)$ is modified into $\oplus_\pm \,c_{\pm}\cdot \mw^1_1(\pm 1)$
 (the restriction to $\vi=\im (h|_{\vi}-1)\subset \widetilde{U}$).
 Hence, $\mv_{{\rm fib}}$ determines $S$.

Conversely, the matrix $S$ itself almost determines $\mv_{{\rm fib}}$.
The only missing data is to know, how to separate the sum $n-1+c_+$
(which is determined from $S$)
into two pieces $n-1$ and $c_+$.
\end{proof}

\section{Isolated complex analytic surface singularities}\label{S:NSS}

\subsection{Links in isolated  surfaces singularities}\label{s:milnor}

Let $(X,0)$ be a complex analytic isolated  surface singularity (germ).
We fix an embedding of  $(X,0)$ into some $\C^N$.
The \emph{link} of $(X,0)$  is the oriented 3-manifold $M$
 obtained as the  intersection $X\cap S^{2N-1}_\e$, where
$S^{2N-1}_\e$ is a sphere of sufficiently small radius $\e$ and  centered at $0$.
The diffeomorphism type of $M$ does not depend on the choice of the embedding and on the radius of the sphere \cite{LDT,Lo,Milnor-book}.

Assume that $g\colon (X,0)\to(\C,0)$ is the germ of an analytic
function, which determines an isolated singularity $(\{g=0\},0)\subset (X,0)$.
If $\e$ is sufficiently small, then the intersection $L_g:=M\cap\{g=0\}$ is transverse.

\begin{definition}
The pair  $L_g\subset  M$ is called the \emph{link} of the germ $g$ at $0$.
\end{definition}

For a germ $g$ as above one defines two fibrations.
The first one is the {\it  Milnor fibration} (see  \cite{Milnor-book} when
 $X$ is smooth and \cite{Ham} in  the general case).

\begin{proposition}\label{prop:MilnorOnM}
The map $\arg g\colon M\sm L_g\to S^1$ defines an open book decomposition of $(M,L_g)$.
\end{proposition}
In parallel, let
$\eta>0$ be sufficiently small, $D_\eta\subset\C$ be the disk of radius $\e$ centre $0$,
 and  $B_\e \subset\C^N$ be the $\e$--ball around $0$. Then one has the {\it tube filtration}
(see  \cite{LDT3}):

\begin{proposition}\label{prop:MilnorOnB}
The map $g\colon (g^{-1}(D_\eta\setminus 0)\cap B_\e)\to D_\eta
\sm 0$ is a locally trivial fibration.
\end{proposition}

\begin{proposition}\label{prop:equiv} \cite{Milnor-book,LDT,CSS}
The fibrations $\arg g$ of \ref{prop:MilnorOnM} and
the restriction of the fibration of \ref{prop:MilnorOnB} to
$S^1=\partial D_\eta$  are equivalent. In particular,
their fibres are diffeomorphic and the monodromy maps coincide.
\end{proposition}

Take $\Sigma:=(\arg g)^{-1}(1)\subset M $ to be the Seifert surface  of $L_g$  and denote
the components of $L_g$ by  $K_1,\ldots, K_n$. For the pair $\S\subset M$ we will use all the
notation of sections 2 and 3.

The following result proves that
$L_g\subset M$ is a 0--link in the sense of Definition~\ref{def:link}.

\begin{lemma}\label{lem:nullhomog}
Each component $K_i$ of $L_g$ represents $0\in H_1(M;\Q)$.
\end{lemma}
\begin{proof} In general it is not true that there exist analytic germs
$g_i:(X,0)\to (\C,0)$ ($1\leq i\leq n$), such that the link of
$g_i$ is $K_i\subset M$. However, if we allow to modify the
analytic structure supported on the topological type of $(X,0)$
(that is, if we keep the pair  $(M,L_g)$ up to an isotopy,
but we change the analytic structure
into some $(X_i,0)$), then such a germ $g_i:(X_i,0)\to (\C,0)$ exists;
see \cite{NP}, or page 3 of \cite{NNP}.
Then the Milnor fiber $\Sigma_i\subset M$ of $g_i$ satisfies
$\p\Sigma_i=K_i$.
\end{proof}

\begin{lemma}\label{lem:vbneg}
$L_g$ is special fibred in the sense of Definition~\ref{def:link}. In particular,
the form $\clk$ on $\vb$ is negative definite (cf. Lemma \ref{lem:nondeg}).
\end{lemma}
\begin{proof}
We need to show that $\clk(K_i,K_j)>0$ for any $ i<j$.
By resolution of singularities, the pair $g^{-1}(0)\subset X$ has an embedded resolution, hence
the pair $L_g\subset M$ has a plumbing representation, where each $K_i$ is represented by an
arrowhead. Let $\Sigma_i$ be the Seifert surface of $K_i$ provided by the
Milnor fiber of $g_i:(X_i,0)\to (\C,0)$ (identified topologically,
cf. the proof of Lemma \ref{lem:nullhomog}).
If the arrowhead associated with $K_j$ is supported by the
vertex $v_j$, then $\clk(K_i,K_j)=\S_i\cdot K_j$ is exactly the multiplicity of the germ
$g_i$ along the exceptional divisor marked by $v_j$. This is a positive integer
(which can be identified combinatorially from the plumbing graph of $(M,L_g)$).
\end{proof}

Finally, we verify that $L_g$ is simple in the sense of Definition~\ref{def:simple}.

\begin{proposition}[The Monodromy Theorem \expandafter  \cite{EN,LDT2}]\label{prop:monod}
The eigenvalues of monodromy $h$ are roots of unity, and $h$ does
 not have any Jordan blocks of size $\ge 3$.
\end{proposition}

\begin{corollary}
The fibred link $(M,L_g,\arg g)$ is a simple fibred link.
\end{corollary}

In particular, its fractured HVS $\mv_{{\rm frct}}$ and the mended HVS $\mvm$ are
 well-defined, and $\mvm=\mv_{{\rm fib}}$.  The analyticity implies one more restriction for them.
It is the generalization of \cite[(3.2)]{Neu-inv} (valid for plane curves and $\lambda\not=1$),
\cite[6.14]{Nem2} (valid for arbitrary hypersurfaces and $\lambda\not=1$).

\begin{proposition}\label{prop:twist}
For any eigenvalue $\lambda$ the block $\mv^2_\lambda(-1)$
  does not appear in $\mv_{{\rm fib}}$.
\end{proposition}
\begin{proof}
Consider first the case $\lambda=1$. Using the terminology of \cite[Sec. 13]{EN},
the twist of the monodromy is nonpositive. This follows similarly as in \cite{EN},
 since it involves only
local computation regarding local analytic germs of type
$f(x,y)=x^ay^b$ with $a,b$ positive integers and $(x,y)\in (\C^2,0)$ corresponding to the
edges of an embedded resolution graph.
This means, that if $\alpha\in H_1(\S,\partial \S)$ then
$\alpha(\Var\alpha)$ should be nonpositive (here we use for $\alpha(\beta)$
the notation of the proof of Theorem \ref{th:mendfib}, regarding the pairing
$(\alpha,\beta)\in H_1(\S,\partial\S)\otimes H_1(\S)$).
 Since $\alpha(\beta)=S(\Var\alpha,\beta)$ (cf. the proof of  \ref{th:mendfib})
  we get that  $\alpha(\Var\alpha)=  S(\Var\alpha,\Var\alpha)$.
 Let us consider
  $\mv^2_1(\pm 1)$. It is the structure
 $$\Big( \C^2;\, \bp  0&\mp 1\\ \pm 1 & 0\ep, \bp 1&1\\ 0&1\ep, \bp \mp 1& 0\\ 0&0\ep\Big).$$
Suppose $\mv^2_1(\pm 1)$ is a summand of $\mv_{{\rm fib}}$, let
$\alpha$ be in $\C^2$ corresponding to $\mv^2_1(\pm 1)$, then $\Var\alpha$ is a vector of type $(a,0)$ for some $a\in\C$
and then $\alpha(\Var\alpha)=\mp a^2$.
This is nonpositive for
any $\alpha$ if and only if from $\mp$ we take the minus sign.

In the literature, there is another test for the sign of the twist which works uniformly for
any $\lambda$. In \cite{Neu-inv} the notations are the following: Neumann's $L$ is our $S^t$,
while his $S$ is our $b^t$. (This can be verified by identifying his identities with ours from
Definition~\ref{def:HVS}.) Then, for any $\lambda$, an $N$--root of unity, the test from pages 228--229 from
\cite{Neu-inv} requires that $S(\alpha, (h^N-1)\alpha)$ is nonpositive. This, in our language,
means that $b^t_\pm(J_2-1)$ must have on the diagonal nonpositive entries
(where $b_\pm$ is the $b$--operator of $\mv^2_\lambda(\pm1)$,
or of $\mv^2_1(\pm1)$ given above).
By a computation, $b^t_\pm(J_2-1)={\tiny \bp  0& 0\\ 0& \mp1\ep}$,
hence in $\mp1$ only the $-$ sign is allowed.
\end{proof}
\begin{remark}\label{rem:graph} As a corollary of Proposition \ref{prop:twist}, the
components $\mvf\oplus \mvb\oplus \widetilde{\mv}$ of  $(\mv_{{\rm fib}})_{=1}$
(generalized eigenspace for $\lambda=1$) are the following:
$\mvf$ is a direct sum of $\frac12\dim\vf$ copies of $\mv^1_1(+1)\oplus \mv^1_1(-1)$,
 $\mvb=(n-1)\cdot \mw^1_1(+1)$, and  $\widetilde{\mv}$ is a direct sum of $\dim \vi$ copies of
  $\mv^2_1(+ 1)$. All these ranks can be read from the
  dual embedded resolution graph of $(M,L_g)$.

Indeed, if $\Gamma$ is the (abstract) dual resolution graph of $(X,0)$, let $c(\Gamma)$ be the
number of independent cycles in $\Gamma$ (that is, the first Betti number of the topological
realization of $\Gamma$), and let $g(\Gamma)$ be the sum of all genus decorations of the vertices.
Then, by   \cite{NS}, $c(\Gamma)=\dim \vi$ (which equals the number of
$2\times 2$-Jordan blocks with eigenvalue one), and $\dim \vf=2g(\Gamma)$.
In particular, all these numbers are independent of the choice of the germ $g$
(that is, of the choice of the link $L_g$).
Moreover, $\dim H_1(M)=2g(\Gamma)+c(\Gamma)$.

On the other hand, $n$  obviously is the number of arrowheads
of the graph of $(M,L_g)$.

Regarding the notations of the proof of Theorem \ref{th:mendfib}, in this analytic case
$c_-=0$, and $c_+=c(\Gamma)$. Hence $(\mv_{{\rm frct}})_{=1}$ consists of $n-1+c(\Gamma)$ copies of
$\mw^1_1(+1)$. Hence, the last statement of Theorem  \ref{th:mendfib} can be reformulated
also as follows: $S$ and $c(\Gamma)$ determine $\mv_{{\rm fib}}$.

In particular, if $\Gamma$ is a tree, then $S$ and $\mv_{{\rm fib}}$ determine each other.
\end{remark}

\subsection{Mixed Hodge Structures on the vanishing (co)homology of $g$}\label{ss:MHS}\
If $g$ is an isolated hypersurface singularity (in any dimension) then
the cohomology of its Milnor fiber carries a mixed Hodge structure by the work
of Steenbrink and Varchenko. The structure is compatible with the monodromy action
(the semisimple and the nilpotent parts are morphisms of type $(0,0)$ and $(-1,-1)$
respectively), and has several polarization properties induced by the intersection and variation
forms. Steenbrink and Varchenko considered also the associated spectrum, which are rational numbers
$\alpha$, one number for each eigenvalue $\lambda=e^{2\pi i\alpha}$ of the monodromy, such that
the choice of $\alpha$ reflects the position in the Hodge filtration of the corresponding eigenvector.

  The more general case when $(X,0)$ is a space germ with an
isolated singularity, and $g:(X,0)\to (\C,0)$ is an analytic function germ which also defines an
isolated singularity,  is treated in \cite{Ste}. In this case, if $\dim (X,0)=2$, then the spectrum
$\Sp_{{\rm MHS}}$ is situated in
the interval $(0,2]$ (or shifted to  $(-1,1]$, but here we prefer the first version).
For precise definitions and particular cases see the articles of Steenbrink and Varchenko in the
present bibliography (e.g. \cite[(5.3)]{Stee} or \cite{Stsp}), and also their references.

In fact, the (co)homologies of the link $M$ itself carry mixed Hodge
structure as well (see e.g. \cite{Ste}). For example, if $\dim
(X,0)=2$, then $\dim\Gr^W_{-1}H_1(M)=2g(\Gamma)$ and $\dim\Gr^W_0H_1(M)=c(\Gamma)$ (in $H^1(M)$ the weights are $+1$ and $0$).
The Hodge numbers for $H_1(M)$ are $h_{-1,0}=h_{0,-1}=g(\Gamma)$ and $h_{0,0}=c(\Gamma)$. 
Moreover, the natural geometric exact sequences (like (\ref{ex:VAR}))
are compatible with MHS (see e.g. the proof of Theorem~\ref{th:specs}). 

If $H_1(M,\Q)=0$  then $\Sp_{{\rm MHS}}$ is symmetric with respect to 1, see  \cite{SSS}.
Hence, in this case $\Sp_{{\rm MHS}}\subset (0,2)$. However, in general,  $\Sp_{{\rm MHS}}\cap\Z$
fails to be symmetric, see below.

\bigskip

In our approach, one can consider the fractured HVS $\mv_{{\rm frct}}$ and its spectrum
$\Sp_{{\rm frct}}\subset (0,2]$ determined  from  $\mv_{{\rm frct}}$ as in Section~\ref{s:specHVS}.

\begin{theorem}\label{th:specs} \

\begin{itemize}
\item[(a)]   $\Sp_{{\rm MHS}}\setminus \Z = \Sp_{{\rm frct}}\setminus \Z$. In particular,
they are both symmetric with respect to 1. Hence $\Sp_{{\rm frct}}$ is also symmetric.

\item[(b)] In $\Sp_{{\rm MHS}}$ the spectral number 2
appears with multiplicity $c(\Gamma)+g(\Gamma)$, while  1 with multiplicity
 $c(\Gamma)+g(\Gamma)+n-1$.

\item[(c)] All integral spectral numbers of  $\mv_{{\rm frct}}$ are concentrated at 1
with multiplicity  $c(\Gamma)+n-1$.

\item[(d)] The spectrum $\Sp_{{\rm MHS}}$ coincides with the spectrum of $\mv_{{\rm fib}}$ (hence also
with the spectrum of $\mv_{{\rm mend}}$, by Theorem~\ref{th:mendfib}).

\item[(e)] In particular,
$\Sp_{{\rm MHS}}=\Sp_{{\rm frct}} + g(\Gamma)\cdot \{1,2\}+ c(\Gamma)\cdot \{2\}$.

\end{itemize}
\end{theorem}
\begin{proof} Here all the spaces are considered with complex coefficients. 
As the monodromy preserves the decomposition
$H_1(M)=\vn\oplus\vf\oplus\vb\oplus(\vbb\oplus\vi)$, the spectrum of MHS is a union of contributions on $\vn$, $\vf$, $\vb$ and $\vbb\oplus\vi$.
On $\vn\oplus\vf$ the intersection form is non--degenerate, so the
polarization property of the MHS (as in \cite[Section 6]{Nem2}) shows that the spectrum of the MHS agrees with the spectrum $\Spn\cup\Spf$. This shows (a). Notice that $\vf$ is the sum of the same amount of
copies of blocks with different polarizations (signs), hence
$\Spf$ is a union of the same amount of copies of $\{1\}$ and $\{2\}$.

On $\vi\oplus\vbb$, the monodromy is the union of two--dimensional Jordan blocks with eigenvalue~1. Each Jordan block corresponds to either $\mv^2_1(1)$,
or $\mv^2_1(-1)$, but the contribution of both structures to the spectrum is the same: each contributes with $\{1,2\}$.
Indeed, the nilpotent monodromy operator shifts the Hodge filtration by $-1$; in particular
$\vi$ contributes with spectral numbers 1 and $\vbb$ with 2.

The contribution of $\vb$ to $\Sp_{\MHS}$ follows from an extension of the argument in \cite[Theorem 2.1]{SSS} (the term $(\#\mathcal{A}-1)(0,1)$
in that article corresponds to the element $(n-1)\cdot\{1\}$ in $\Sp_{\MHS}$).
The above discussion (see also
Remark~\ref{rem:graph}) shows that $\Sp_{\MHS}$ agrees with the spectrum of $\mv_{{\rm fib}}$.
So (b) and (d) are also proved.

(c) For $\lambda=1$
we have only blocks of type $\mw^1_1(+1)$ (cf. Proposition~\ref{prop:twist}); then use  Definition~\ref{def:spectrum}.

Part (e) is a consequence of (a)--(d) and the comparison of $\mvc$ with $\mv_{{\rm fib}}$ in Remark~\ref{rem:graph}.

\vspace{2mm}

An alternating way to check the Hodge types of the blocks $\vb$ and $\vf\oplus \vbb$ is via exact sequences. 
There are two main exact sequences (usually written in cohomology and compact support cohomology), both of them being sequences of mixed Hodge structures, see \cite[(2.6)(a)-(b)]{Ste}. 
The first, written in our language, is
$$ 0\to \C\to H_1(\partial \Sigma)\to (H_1(\Sigma))_1\stackrel{j}{\to}(H_1(\Sigma,\partial \Sigma))_1\to H_1(\partial \Sigma)\to \C\to 0,$$
where (by duality) $j$ can be identified with our $b_1$. In particular,  the kernel of $b_1$ (hence $\vb$ too)
supports a  mixed Hodge structure,
as the link of the curve singularity $X\cap \{f=0\}$. For dimensional reason
it supports only one  Hodge type, which  is the same as for curves
sitting on surface singularities with rational homology sphere links, or even as for plane curve singularities. Hence they have the same type of spectrum contributions, namely 1.  

To identify the term $\vf\oplus \vbb\simeq {\rm coker}({\rm Var})$ we consider the `variation exact sequence':
$$0\to H_2(M)\to H_1(\Sigma,\partial \Sigma)\stackrel{{\rm Var}}{\longrightarrow}H_1(\Sigma)\to H_1(M)\to 0$$
and use the Hodge types of $M$, cf. Section~\ref{ss:MHS}. 
\end{proof}
We emphasize again that $\Sp_{{\rm frct}}$ can be connected with the signatures of
 $\mv_{{\rm frct}}$ by Lemma
\ref{lem:sigandspec}. They agree with the signatures of the Seifert form because of the following lemma.
\begin{lemma}\label{lem:signature}
The fractured signature of the link $L_g$ (defines via the Seifert matrix,
cf. Definition~\ref{def:fracturedsig}) is equal to the signature of $\mv_{{\rm frct}}$
(cf.  Definition~\ref{def:signatureofHVS}).
\end{lemma}
\begin{proof}
Note that  $V=(S^t)^{-1}$. But then  we have
\[
S\left((1-z)V+(1-\ol{z})V^t\right)S^t=(1-z)S+(1-\ol{z})S^t,
\]
hence the two forms $(1-z)V+(1-\ol{z})V^t$ and  $(1-z)S+(1-\ol{z})S^t$
are congruent. Their signatures coincide.
\end{proof}

\subsection{Deformations of singularities}\label{s:notation}
In this section we establish notation, which will allow us to formulate and prove Theorem~\ref{thm:semic}.

\begin{definition}
A \emph{deformation} of an isolated singularity $g_0:(X,0)\to (\C,0)$
is a complex 3-dimensional variety $\X\subset\C^{N}\times D$ (where $D$
is a small disk in $\C$ centered at the origin)
together with an analytic function $G\colon\X\to\C$ and a projection $\pi\colon\X\to D$ such that:
\begin{itemize}
\item[$\bullet$] $\pi$ is a flat morphism;
\item[$\bullet$] for $t\in D$, the inverse image $X_t:=\pi^{-1}(t)$ is a surface with isolated singularities;
\item[$\bullet$] the function $g_t=G|_{X_t}$ has only isolated singularities;
\item[$\bullet$] the central fiber $X_0$ has a single singularity $x_0$ and $g_0$ is regular away from $x_0$.
\end{itemize}
\end{definition}

Given such a deformation, let us choose a small ball $B_0\subset\C^{N}$ and put $S_0=\p B_0$.
Suppose that the ball is such that $X_0\cap S_0$ is the link
$M_0$ of the singularity $x_0\in X_0$. Shrinking $B_0$ if necessary,
we can assume that $g_0^{-1}(0)\cap S_0$ is the link of singularity of $g_0$ at $x_0$. We
shall denote this link by $L_0$.


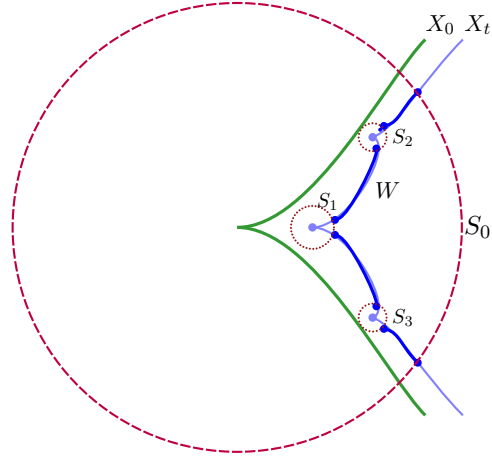
\begin{figure}
\begin{pspicture}(-6,-3)(6,3)
\psbezier[linewidth=1.2pt,linecolor=darkgreen](2.5,2.5)(2,2)(1,0)(0,0)\rput(2.7,2.7){\psscalebox{0.8}{$X_0$}}
\psbezier[linewidth=1.2pt,linecolor=darkgreen](2.5,-2.5)(2,-2)(1,0)(0,0)
\psbezier[linewidth=0.8pt,linecolor=moderateblue](3,2.5)(2.5,2)(2.1,1.3)(1.8,1.2)\rput(3.2,2.7){\psscalebox{0.8}{$X_t$}}
\psbezier[linewidth=0.8pt,linecolor=moderateblue](1.8,1.2)(2.1,1.1)(1.5,0)(1,0)
\psbezier[linewidth=0.8pt,linecolor=moderateblue](3,-2.5)(2.5,-2)(2.1,-1.3)(1.8,-1.2)
\psbezier[linewidth=0.8pt,linecolor=moderateblue](1.8,-1.2)(2.1,-1.1)(1.5,0)(1,0)
\pscircle[fillcolor=moderateblue,fillstyle=solid,linestyle=none](1,0){0.07}
\pscircle[fillcolor=moderateblue,fillstyle=solid,linestyle=none](1.8,1.2){0.07}
\pscircle[fillcolor=moderateblue,fillstyle=solid,linestyle=none](1.8,-1.2){0.07}
\psbezier[linecolor=blue,linewidth=1.2pt,dotsize=3pt]{*-*}(2.4,1.8)(2.2,1.6)(2.15,1.4)(1.95,1.35)
\psbezier[linecolor=blue,linewidth=1.2pt,dotsize=3pt]{*-*}(1.85,1.05)(1.85,0.9)(1.5,0.17)(1.3,0.1)
\psbezier[linecolor=blue,linewidth=1.2pt,dotsize=3pt]{*-*}(2.4,-1.8)(2.2,-1.6)(2.15,-1.4)(1.95,-1.35)
\psbezier[linecolor=blue,linewidth=1.2pt,dotsize=3pt]{*-*}(1.85,-1.05)(1.85,-0.9)(1.5,-0.17)(1.3,-0.1)
\rput(2,0.5){\psscalebox{0.8}{$W$}}

\pscircle[fillcolor=blue,fillstyle=solid,linestyle=none](1.9,1.3){0.04}
\pscircle[linestyle=dotted,linecolor=crimson,dotsep=0.5pt](1,0){0.3}\rput(1.2,0.35){\psscalebox{0.7}{$S_1$}}
\pscircle[linestyle=dotted,linecolor=crimson,dotsep=0.5pt](1.8,1.2){0.2}\rput(2.2,1.2){\psscalebox{0.7}{$S_2$}}
\pscircle[linestyle=dotted,linecolor=crimson,dotsep=0.5pt](1.8,-1.2){0.2}\rput(2.2,-1.2){\psscalebox{0.7}{$S_3$}}
\pscircle[linestyle=dashed,linecolor=purple, dash=4pt 1pt](0,0){3}\rput(3.2,0){\psscalebox{0.9}{$S_0$}}
\end{pspicture}
\caption{Deformation of the singularity. We denote $S_i=\p B_i$.}\label{fig:deform}
\end{figure}

Let now $t\in D\sm\{0\}$ be sufficiently small. Then $X_t\cap S_0\cong M_0$.
Furthermore, by choosing $t$ small enough we can guarantee that
$(X_t\cap S_0,g_t^{-1}(0)\cap X_t\cap S_0)\cong (M_0,L_0)$ as pairs.
Let $x_1,\dots,x_k$ be the critical points of $g_t$ on $X_t$. (If
$x\in X_t$ is a singular point of $X_t$ and $g_t(x)=0$, then $x$ has to be
considered as a critical point of $g_t$.)
Let $B_1,\dots,B_k$
be small pairwise disjoint balls near $x_1,\dots,x_k$ such that $B_i\subset B_0$
and the pair $(M_i,L_i):=(\p B_i\cap X_t,\p B_i\cap X_t\cap g_t^{-1}(0))$
is the link of the singularity of $g_t$ at $x_i$.
Finally let
\[W=\ol{B_0\cap X_t\sm(B_1\cup\cdots \cup B_k)}.\]
Then $W$ is a cobordism between a disjoint union $M_1\cup\cdots \cup M_k$ and $M_0$.
In general, $W$ can have a finite number of singular points: these are
all those singular points of $X_t$ where $g_t$ does not vanish. See Figure~\ref{fig:deform}.

Let us consider the map $\arg g_t\colon W\to S^1$. This is a surjection and let
 us pick a regular value $\delta$ such
that $\O:=\arg g_t^{-1}(\delta)$ omits all the singular points of $W$.
 We have the following observation.

\begin{lemma}
For any  $i=0,\dots,k$ the intersection $\O\cap\p B_i$ is the Seifert surface $\S_i$ for $L_i$ cut out by its
Milnor open book.
\end{lemma}

Let $Y=\ol{\p\O\sm\bigcup_{i=0}^k\S_i}$. Observe that $Y=g_t^{-1}(0)\cap W$.
Let also $Z=Y\cup\S_1\cup\dots\cup\S_k$. (This has some `corners' along $\partial Y$, but they can be smoothed.)

\begin{lemma}\label{lem:Zisdiffeo}
The manifold $Z$ is diffeomorphic to a Seifert surface $\S_0$.
\end{lemma}
\begin{proof}
By Proposition~\ref{prop:equiv}, $\S_0\cong g_0^{-1}(\delta)\cap B_0$ if $\delta\in\C\sm\{0\}$
is sufficiently small. Then, since $t$ is sufficiently close
to $0$, $\S_0\cong g_t^{-1}(\delta)\cap B_0$. Now
\[g_t^{-1}(\delta)\cap B_0=\big(g_t^{-1}(\delta)\cap W\big)
\cup\bigcup_{i=1}^k \big(g_t^{-1}(\delta)\cap B_i\big).\]
Applying Proposition~\ref{prop:equiv} again, we have $g_t^{-1}(\delta)\cap B_i\cong\S_i$.
On the other hand, since $\delta$ is very small and $g_t^{-1}(0)$
has no singular points, we have $Y=g_t^{-1}(0)\cap W\cong g_t^{-1}(\delta)\cap W$.
\end{proof}

\subsection{Semicontinuity of $\Spc$}\label{s:semic}

Given the notation introduced in Section~\ref{s:notation} we are ready to formulate and prove
the next semicontinuity result regarding the spectrum.  $\Spc^i$ denotes the spectrum associated with the
corresponding local fractured  HVS, $i=0,\ldots, k$.

\begin{theorem}\label{thm:semic}
If $s\in[0,1]$ is such that $z=e^{2\pi is}$ is not an eigenvalue of the monodromy of $L_0$, then
\begin{align*}
|\Spc^0\cap(s,s+1)|\ge& \sum_{i=1}^k|\Spc^i\cap(s,s+1)|-\Irr_2-\Irr_1\\
|\Spc^0\sm[s,s+1]|\ge& \sum_{i=1}^k|\Spc^i\sm[s,s+1]|-\Irr_2-\Irr_1,
\end{align*}
where
\[\Irr_1=\dim\ker(H_1(M_0\cup M_1\cup\dots\cup M_k)\to H_1(W))-\sum_{i=1}^k b_1(M_i).\]
\end{theorem}

\begin{proof}
The pair $(W,\O)$ is a Seifert cobordism of links $(M_0,L_0)$ and
$(M_1,L_1)\sqcup\dots\sqcup(M_k,L_k)$. For $i=0,\dots,k$, let
$\s_i(z)$ denotes the fractured signature of the link $L_i$. Suppose first that
$z$ is chosen so that $z$ is not an eigenvalue of the monodromy of any $h_i$. Since the links $(M_0,L_0)$,\dots,$(M_k,L_k)$ are
simple fibred, we have $n_0(z)=n_1(z)=\dots=0$.

Theorem~\ref{thm:signatures} applied in this situation yields

\begin{multline*}
\left|\sum_{i=1}^k\s_i(z)-\s_0(z)\right|\le
\sum_{i=1}^k(\dim \vs_i-2b_1(\S_i))+b_1(\p\O)+\dim \vs_0-2b_1(\S_0)+\\
+2\Irr_2+2\dim\ker(H_1(M_0\cup M_1\cup\dots\cup M_k)\to H_1(W)).
\end{multline*}

Here $\vs_i=\ker(H_1(\S_i)\to H_1(M_i))$ and $\dim\vs_i$ is the size of the fractured Seifert matrix for $M_i$.
Therefore  $\dim \vs_i-2b_1(\S_i)=-\dim \vs_i-2b_1(M_i)$.
On the other hand, by Lemma~\ref{lem:Zisdiffeo},  we have
$\p\O\cong\S_0\cup\S_0$, hence
$b_1(\p\O)=2 b_1(\S_0)$.
We obtain.

\begin{equation*}\label{eq:firstsemic}
\begin{split}
&\left|\sum_{i=1}^k\s_i(z)-\s_0(z)\right|+\sum_{i=1}^k\dim \vs_i-\dim \vs_0\le\\
&\le 2\Irr_2+2\dim\ker(H_1(M_0\cup M_1\cup\dots\cup M_k)\to H_1(W))-2\sum_{i=1}^k b_1(M_i).
\end{split}
\end{equation*}

The proof now follows from Lemma~\ref{lem:sigandspec}. It remains to deal with the
case where $z$ is an eigenvalue of $h_j$ for some $j>0$.
This is done by choosing $z'$ sufficiently close to $z$ and using the result for $z'$.
The argument is as in \cite[Section 4.1]{BNR}, we omit here the details.
\end{proof}

\subsection{Special cases of Theorem~\ref{thm:semic}}

Theorem~\ref{thm:semic} is stated in a rather general form, $X_t$ might have many singular points, $W$ itself
is allowed to be singular. Sometimes it is more convenient to have some special cases. We begin with the following lemma
\begin{lemma}
We have
\[\Irr_1+\Irr_2=\dim H_2(W,M)-\sum_{i=1}^k b_1(M_i),\]
where $M=\p W=M_0\sqcup M_1\sqcup\dots\sqcup M_k$.
\end{lemma}
\begin{proof}
By the long exact sequence of the pair $(W,M)$ we obtain
\[\dim\coker(H_2(M)\to H_2(W))+\dim\ker(H_1(M)\to H_1(W))=\dim H_2(W,M).\]
\end{proof}

\begin{proposition}\label{prop:irr=0}
Suppose that $M_0\cong M_1$ are rational homology cobordant and $M_2\cong\dots\cong M_k\cong S^3$. Suppose additionally that $W$ is built from a rational
$H$--cobordism (that is the inclusions  $M_0\hookrightarrow W'$ and $M_1\hookrightarrow W'$ induce isomorphism on rational homologies)
$W'$ between $M_0$ and $M_1$ by removing $k-1$ balls, then $\Irr_1+\Irr_2=0$.
\end{proposition}
\begin{proof}
Clearly $H_2(W',M)\cong H_2(W,M)$. Furthermore $\dim H_2(W,M)=b_1(M_0)=b_1(M_1)$. The statement follows by definition.
\end{proof}
\begin{corollary}
If $\X$ is a trivial deformation, that is $X_t\cong X_0$, then $\Irr_1+\Irr_2=0$.
\end{corollary}
\begin{proof}
We can choose $M_1$ to be equal to $M_0$. Then $W$ is obtained from $M_0\times [0,1]$ by removing a finite number of 4-balls; these balls are neighbourhoods
of the critical points of $g_t$ on $g_t^{-1}(0)$. Thus $W$ satisfies the assumptions of Proposition~\ref{prop:irr=0}.
\end{proof}

\subsection{Semicontinuity results for $\SpM$}

Using Theorem~\ref{th:specs}(e) we can now deduce semicontinuity property for $\SpM$ from Theorem~\ref{thm:semic}. For $i=0,1,\ldots,k$
let $c_i$ and $g_i$ be the quantities $c(\Gamma)$ and $g(\Gamma)$ corresponding to $M_i$, as in Remark~\ref{rem:graph}. That is $c_i$
is the number of independent cycles in the graph $\Gamma_i$ representing the link $M_i$, while $g_i$ is the sum of all genus decorations of
the vertices. Let us set
\begin{equation*}
\Delta_1=c_0-\sum_{i=1}^kc_i, \ \ \ \ \
\Delta_2=c_0+g_0-\sum_{i=1}^k(c_i+g_i).
\end{equation*}
Then one has the following result.

\begin{theorem}\label{thm:semic-MHS}
If $s\in[0,1]$ is such that $z=e^{2\pi is}$ is not an eigenvalue of the monodromy operator of $L_0$, then
\begin{align*}
|\SpM^0\cap(s,s+1)|\ge& \sum_{i=1}^k|\SpM^i\cap(s,s+1)|-\Irr_2-\Irr_1+\Delta_1\\
|\SpM^0\sm[s,s+1]|\ge& \sum_{i=1}^k|\SpM^i\sm[s,s+1]|-\Irr_2-\Irr_1+\Delta_2.
\end{align*}
\end{theorem}
\begin{proof}
Suppose $s\in(0,1)$. We use the fact that $|(\SpM^i\sm\Spc^i)\cap(s,s+1)|=c_i$ and $|(\SpM^i\sm\Spc^i)\sm(s,s+1)|=c_i+g_i$,
which follows from Theorem~\ref{th:specs}(d).

If $s=0$ or $s=1$, then the assumptions imply that $1$ is not an eigenvalue of the monodromy operator of $L_0$, in particular, neither 1 nor 2
are in $\SpM^0$. By Remark~\ref{rem:graph} we infer that $c_0=g_0=0$, hence $\Delta_1,\Delta_2\le 0$. Clearly, $\SpM^i\cap(0,1)=\Spc^i\cap(0,1)$,
so the first inequality holds if $s=0,1$. If $s=0$, we have $\SpM^i\sm[0,1]=\Spc^i\sm[0,1]\cap (g_i+c_i)\cdot\{2\}$, so the second inequality in that case
follows from the case $s\in(0,1)$. As $\SpM^i\sm[1,2]=\Spc^i\sm[1,2]$ and $\Delta_2\le 0$, we infer that the second inequality holds for $s=1$ as well.
\end{proof}
As a corollary, we prove the semicontinuity in the case of Proposition~\ref{prop:irr=0}.

\begin{proposition}
Under the assumptions of Proposition~\ref{prop:irr=0}, for instance if $X_t\cong X_0$, we have $\Delta_1=\Delta_2=0$. Thus for any $s\in[0,1]$
such that $e^{2\pi i s}$ is not an eigenvalue of the monodromy operator of $L_0$ we have
\begin{align*}
|\SpM^0\cap(s,s+1)|\ge& \sum_{i=1}^k|\SpM^i\cap(s,s+1)|\\
|\SpM^0\sm[s,s+1]|\ge& \sum_{i=1}^k|\SpM^i\sm[s,s+1]|.
\end{align*}
\end{proposition}

\end{document}